\documentclass[11pt]{article}
\usepackage{amsfonts,amsmath,amsthm,amssymb,enumerate,color}
\usepackage[makeroom]{cancel}
\usepackage{authblk}

\usepackage{todonotes}

\usepackage[shortlabels]{enumitem}
\setlist[enumerate,1]{wide, labelindent=0pt,label={\upshape(\roman*)}}

\numberwithin{equation}{section}
\newcommand{\beq}{\begin{equation}}
	\newcommand{\eeq}{\end{equation}}
\newcommand{\bea}{\begin{eqnarray}}
	\newcommand{\eea}{\end{eqnarray}}
\newcommand{\beas}{\begin{eqnarray*}}
	\newcommand{\eeas}{\end{eqnarray*}}

\newtheorem{theorem}{Theorem}[section]

\newtheorem{proposition}[theorem]{Proposition}

\newtheorem{corollary}[theorem]{Corollary}
\newtheorem{lemma}[theorem]{Lemma}
\newtheorem*{localizationlemma}{Localization lemma}
\newtheorem{remark}[theorem]{Remark}
\newtheorem{example}[theorem]{Example}
\newtheorem{examples}[theorem]{Examples}
\newtheorem{foo}[theorem]{Remarks}

\newcommand{\abs}[1]{\left|#1\right|}     

\newcommand{\bM}{\mathbb M}

\newcommand{\Ho}{\mathcal H}
\newcommand{\V}{\mathcal V}
\newcommand{\hor}{\mathcal H}
\newcommand{\ver}{\mathcal V}

\newcommand{\M}{\mathbb M}
\newcommand{\Tor}{\mathrm{Tor}}

\title{Sub-Laplacian generalized curvature dimension inequalities on Riemannian foliations}

\author{Fabrice Baudoin
	, Guang Yang}

\begin{document}
	
	\maketitle
	
	\newcommand\blfootnote[1]{%
		\begingroup
		\renewcommand\thefootnote{}\footnote{#1}%
		\addtocounter{footnote}{-1}%
		\endgroup
	}

	\begin{abstract}
    We develop a Bochner theory and Bakry-\'Emery calculus for horizontal Laplacians associated with general Riemannian foliations. No bundle-like assumption on the metric, nor any total geodesicity or minimality condition on the leaves, is imposed. Using a metric connection adapted to the horizontal–vertical splitting, we derive explicit Bochner formulas for the horizontal Laplacian acting on horizontal and vertical gradients, as well as a unified identity for the full gradient. These formulas involve horizontal Ricci curvature, torsion, and mean curvature terms intrinsic to the foliated structure.
From these identities, we establish generalized curvature–dimension inequalities, extending earlier results in sub-Riemannian geometry. As applications, we obtain horizontal Laplacian comparison theorems, Bonnet–Myers type compactness results with explicit diameter bounds, stochastic completeness, first eigenvalue estimates and gradient and regularization estimates for the horizontal heat semigroup. The framework applies, in particular, to contact manifolds and Carnot groups of arbitrary step.

\vspace{15pt}


	\end{abstract}

\vspace{10pt}
\noindent
\begin{minipage}{\textwidth}
    \small
    \textbf{Fabrice Baudoin:} \\
    Department of Mathematics, Aarhus University, Denmark \\
    Email: fbaudoin@math.au.dk \\
    Research partially supported by grant 10.46540/4283-00175B from Independent Research Fund Denmark and by the Villum Investigator grant \emph{Stochastic Analysis in Aarhus}. F.B. also acknowledges funding from the European Research Council (ERC) under the European Union’s Horizon Europe research and innovation programme (RanGe project, Grant Agreement No. 101199772).
\end{minipage}

\vspace{10pt}

\noindent
\begin{minipage}{\textwidth}
    \small
    \textbf{Guang Yang:} \\
    Department of Mathematics, Southern University of Science and Technology, China \\
    Email: yangg7@sustech.edu.cn
\end{minipage}
	
	\newpage
	

    \section{Introduction}

Subelliptic operators arising from geometric structures with bracket-generating distributions play a central role in analysis (\cite{JerisonSanchezCalle1987,NagelSteinWainger1985}), geometry (\cite{AgrachevBarilariBoscain2019}), and probability (\cite{Thalmaier-lectures}). Among these, horizontal Laplacians associated with foliations constitute a flexible framework that allows for the use of Riemannian techniques to capture genuine sub-Riemannian phenomena. Those horizontal Laplacians naturally arise in a wide range of settings, including contact  manifolds, Carnot groups, and sub-Riemannian manifolds with transverse symmetries.

The purpose of this article is to develop a systematic Bochner–Bakry–Émery theory for horizontal Laplacians on general Riemannian foliations, without assuming that the Riemannian metric is bundle-like, nor that the leaves are totally geodesic or minimal. Removing these classical assumptions introduces new geometric features most notably torsion and mean curvature effects that fundamentally alter the structure of Bochner identities and curvature dimension inequalities. The main goal of this work is to show that, despite this increased complexity, a robust  framework can still be established and exploited to derive meaningful analytic and geometric consequences.

\subsection{Geometric setting}

Let $(M,g)$ be a complete Riemannian manifold endowed with a foliation $\mathcal F$ whose tangent bundle $TM$ splits orthogonally into horizontal and vertical subbundles $$TM=\mathcal H\oplus\mathcal V,$$ 
where $\V$ is tangent to the leaves. In contrast with much of the existing literature, we do not assume that:
\begin{itemize}
    \item the Riemannian metric is bundle-like; 
    \item the leaves are totally geodesic or  minimal.
\end{itemize}
These assumptions, while technically convenient, exclude many natural examples. In particular, Carnot groups of step greater than two and general contact manifolds fall outside the classical framework. In such situations, the interaction between curvature, torsion, and the mean curvature of the leaves plays a crucial role and must be explicitly accounted for.

\medskip

The horizontal Laplacian $\Delta_{\mathcal H}$ considered here is defined as the divergence of the horizontal gradient with respect to the Riemannian volume measure. 
We assume throughout that the horizontal distribution $\mathcal H$ is bracket generating. Under the bracket-generating assumption and completeness of $g$, this operator is locally subelliptic, essentially self-adjoint and generates a sub-Markovian heat semigroup $(P_t)_{t\ge0}$ which admits a heat kernel.

\subsection{Bochner's identities beyond the classical setting}

In the Riemannian case, lower Ricci curvature bounds and their analytic consequences are classically derived from Bochner’s identity and the Bakry–Émery  $\Gamma_2$-calculus, see the monograph \cite{BGL}. Extending this approach to sub-Riemannian and foliated contexts has been the subject of extensive research over the past two decades, see \cite{BaudoinEMS2014,BG17,BW14,GrTh16a, FYWbook}. Early developments focused primarily on Riemannian foliations with bundle-like metrics and totally geodesic leaves, where the horizontal Laplacian enjoys additional symmetry properties and Bochner-type formulas more closely resemble their Riemannian counterparts. The present work addresses this gap by providing  general Bochner's formulas and Bakry-\'Emery calculus for the horizontal Laplacian of arbitrary Riemannian foliations.

\medskip

A key ingredient of our approach is the use of a metric connection $\nabla$ that is adapted to the splitting $TM=\mathcal H\oplus\mathcal V$. This connection, first introduced by Hladky in \cite{Hladky} and then used in the context of foliations in \cite{AGAG,2025arXiv250913276B}, preserves both subbundles and has torsion encoding  the non-integrability of the horizontal distribution and the second fundamental form of the leaves. While the Levi--Civita connection is poorly suited for computations in this setting, the adapted connection allows for a transparent decomposition of second-order quantities into horizontal and vertical components.  Using this connection, the first main contribution of this work is the derivation of explicit Bochner formulas for both horizontal and vertical directions. Specifically, for a smooth function $f$, we express the quantities:
\[ 
\Gamma_2^{\mathcal H}(f):=\frac{1}{2}\Delta_{\mathcal{H}}|\nabla_{\mathcal{H}}f|^{2}-\langle\nabla_{\mathcal{H}}f,\nabla_{\mathcal{H}}\Delta_{\mathcal{H}}f\rangle \quad \text{and} \quad \Gamma_2^{\mathcal V}(f):=\frac{1}{2}\Delta_{\mathcal{H}}|\nabla_{\mathcal{V}}f|^{2}-\langle\nabla_{\mathcal{V}}f,\nabla_{\mathcal{V}}\Delta_{\mathcal{H}}f\rangle \]
in terms of tensors related to the connection $\nabla$. Here $\nabla_\Ho$ is the horizontal gradient and $\nabla_\V$ the vertical one. These formulas incorporate the mean curvature vector field $H$ of the leaves, curvature tensors related to the horizontal distribution and the torsion of the connection. By combining the horizontal and vertical identities, we also obtain a Bochner formula for the full Riemannian gradient.

\subsection{Generalized curvature dimension inequalities}

From the Bochner identities, we establish generalized curvature–dimension inequalities in the sense of Bakry–Émery. 
As a first application, we identify a tensorial quantity $\mathfrak R$ that controls the full gradient Bochner formula for $\Delta_{\mathcal H}$. Under suitable lower bounds involving $\mathfrak R$, the mean curvature, and the horizontal trace of the mixed torsion, we establish inequalities of the form
\begin{align}\label{CD basic intro}
\Gamma_2^{\mathcal H}(f)+\Gamma_2^{\mathcal V}(f) \ge \frac{1}{N}(\Delta_{\mathcal H}f)^2 + K\lvert\nabla f\rvert^2,
\end{align}
for appropriate constants $N$ and $K$. Under stronger curvature bound conditions, we also prove the following one-parameter family of curvature dimension inequalities: For every $\nu >0$
\begin{align}\label{CD intro}
\Gamma^\Ho_2 (f)+ \nu \Gamma_2^\V (f) \ge \frac{1}{N}(\Delta_\Ho f)^2 +\left(\rho_1-\frac{\kappa}{\nu} \right) | \nabla_\Ho f|^2+ (\rho_2 -\rho_3 \nu -\rho_4 \nu^2) | \nabla_\V f|^2.
\end{align}
This generalized curvature dimension inequality is a direct generalization of the  curvature dimension inequality
\[
\Gamma^\Ho_2 (f)+ \nu \Gamma_2^\V (f) \ge \frac{1}{N}(\Delta_\Ho f)^2 +\left(\rho_1-\frac{\kappa}{\nu} \right) | \nabla_\Ho f|^2+ \rho_2 | \nabla_\V f|^2
\]
obtained under the geometric hypotheses of \cite{BG17} and \cite{GrTh16a} for bundle-like and totally geodesic foliations. It is also a generalization of the curvature dimension inequality obtained for contact manifolds in \cite{BW14}. Therefore a novelty in our curvature dimension inequalities is the appearance of the strongly nonlinear term $  - \rho_4 \nu^2$, in front of the vertical gradient norm, which makes the analysis substantially more delicate. The coefficient $\rho_4$ controls contributions from the failure of the metric to be bundle-like as well as mixed curvature and mean curvature terms. Nevertheless, we show that this generalized curvature–dimension framework remains powerful enough to derive a wide range of analytic and geometric results. 

\subsection{Applications}

As applications of the curvature–dimension inequalities, we obtain:

\begin{enumerate}
    \item \textbf{Laplacian Comparison Theorem:} We generalize the horizontal Laplacian comparison theorem for the horizontal Laplacian of the Riemannian distance to settings without the bundle-like condition or minimal leaf assumptions. Removing those conditions therefore makes our result a generalization of the corresponding result found in \cite{BauGr}, see also \cite{2025arXiv250913276B} and \cite{MR4167256}.

    \item \textbf{Bonnet-Myers Type Result:} As a consequence of the Laplacian comparison theorem, we show that under a positive curvature condition, the manifold $M$ must be compact with an explicit diameter bound on the Riemannian diameter of the space. More precisely, we show that \eqref{CD basic intro} with $K>0$ yields
    \[
\mathbf{diam} (M ) \le \pi \sqrt{ \frac{N}{ K}}.
\]

    \item \textbf{Stochastic completeness:} Under \eqref{CD basic intro}, we prove that the horizontal heat semigroup $P_t$ is stochastically complete, i.e. $P_t1=1$.

    \item \textbf{First Eigenvalue Lichnerowicz Type estimates:} In the positive curvature case we obtain lower bounds for the first eigenvalue of the horizontal Laplacian of the type
    \[
    \lambda_1 \ge C
    \]
    where $C>0$ is a constant explicitly depending on curvature parameters. For instance, under \eqref{CD intro}, with $N>1$, $\rho_2>0$, $\rho_3,\rho_4,\kappa\ge0$ and a positive numerator, one can take $C=\frac{\rho_1 \rho_2 - \kappa (\rho_3 +\sqrt{\rho_2\rho_4})}{\left(\frac{N-1}{N}\right)\rho_2+\kappa }$.

    \item \textbf{Heat Kernel Gradient Bounds:} We establish Bakry-Émery type estimates and global regularization estimates for the horizontal heat semigroup $(P_t)_{t\ge0}$. For example,  we obtain the gradient estimate
    \begin{align}\label{BE intro}
        | \nabla P_{t}f|^2+\frac{2}{N}\frac{1-e^{-2Kt}}{2K} (\Delta_{\mathcal{H}} P_{t}f)^2 \le e^{-2Kt}P_t (|\nabla f|^2)
        \end{align}
    and on uniformly step-two generating distributions, we obtain in small times a reverse Poincar\'e type estimate:
    \begin{align}\label{poincare intro}
    |\nabla_{\mathcal{H}}P_t f|^2+ t (\Delta_{\mathcal{H}} P_{t}f)^2 \leq \frac{c}{t} \left( P_t(f^2) - (P_t f)^2 \right). 
    \end{align}
\end{enumerate}

 We note that Li–Yau type gradient estimates for the heat kernel are also  expected to hold in our framework in light of \cite{CaoYAu,Ivanov} and \cite{Dong2} and will possibly be studied in a later work.

\subsection{Structure of the paper}
The paper is organized as follows.
In Section 2, we introduce the geometric framework, define the horizontal Laplacian and the adapted connection, and recall basic analytic properties of the associated heat semigroup.
Section 3 is devoted to the derivation of Bochner formulas and curvature–dimension inequalities.
In Section 4, we present applications, including Laplacian comparison theorems, eigenvalue estimates, and gradient bounds for the horizontal heat semigroup.

 \

\textbf{Notations:}

\begin{itemize}
\item If $M$ is a manifold, $TM$ is the tangent bundle.
\item $\mathcal{L}$ is the Lie derivative
 \item If $\mathcal{W}$ is a vector bundle over $M$, $\mathfrak{X}(\mathcal{W})$ is the set of smooth sections of that bundle.
 \item If $g$ is a Riemannian metric we denote $\left\langle u, v \right\rangle=g(u,v)$, $|u|^2=g(u,u)$.
\end{itemize}

	\section{Preliminaries}
	\subsection{Setup and assumptions}

	Throughout the paper, we consider a smooth connected $n+m$ dimensional manifold $M$ which is equipped with a foliation $\mathcal{F}$ with $m$ dimensional leaves.  We  assume that $M$ is equipped with a complete Riemannian metric $g$. For $x \in M$, $\mathcal{F}_x$ denotes the leaf going through $x$. The sub-bundle $\mathcal{V}$ of the tangent bundle $TM$ formed by vectors tangent to the leaves is referred  to as the set of \emph{vertical directions}. The sub-bundle $\mathcal{H}$ which is normal to $\mathcal{V}$ is referred to as the set of \emph{horizontal directions}. We assume that $\Ho$ is bracket generating.
    
    In this setting, any vector $u \in T_xM$ can  be decomposed as 
	\[
	u=u_\mathcal{H} +u_\mathcal{V}
	\]
	where $u_\mathcal{H}$ (resp. $u_\mathcal{V}$) denotes the orthogonal projection of $u$ onto $\mathcal{H}_x$  (resp. $\mathcal{V}_x$). 

Note that here \emph{we do not assume} that the metric $g$ is bundle-like. Throughout the paper we will denote  the Levi-Civita connection on $(M,g)$ by $D$. 
We refer  to the classical reference  \cite{Tondeur} or the more recent monograph \cite{Gromoll} for an overview of the theory of foliations.
	Some examples of such structures include the following.
 
\begin{example}\label{Example contact}(Contact manifolds)
 Let $(\M,\theta)$ be a $2n+1$-dimensional smooth contact manifold with Reeb vector field $\xi$. The Reeb foliation on $\M$ is given by the orbits of $\xi$.  From  \cite{Sasaki60}, it is always possible to find a Riemannian metric $g$ and a $(1,1)$-tensor field $J$ on $\M$ so that for all vector fields $X, Y$
\begin{equation} \label{JmapKCont}
g(X,\xi)=\theta(X),\quad J^2(X)=-X+\theta (X) \xi, \quad 2g(JX,Y)= d\theta(X,Y).
\end{equation}
The triple $(\M, \theta,g)$ is called a contact Riemannian manifold. Observe that the horizontal distribution $\Ho$ is  the kernel of $\theta$ and that $\Ho$ is bracket-generating because $\theta$ is a contact form and thus non-degenerate. In the Sasakian case, the Reeb foliation is totally geodesic and the the metric is bundle-like.
\end{example}


\begin{example}\label{Carnot example}(Carnot groups)
A large class of examples that also fit our framework is the class of Carnot groups. A Carnot group is a connected, simply connected nilpotent Lie group $G$ whose Lie algebra  $\mathfrak{g}$ admits a stratification
	\begin{align}\label{stratification}
		\mathfrak{g} = V_{1} \oplus V_{2} \oplus \cdots \oplus V_{s},
	\end{align}
	with the properties:
	\begin{enumerate}
		\item $[V_{1}, V_{j}] = V_{j+1}$ for every $1 \leq j < s$,
		\item $[V_{1}, V_{s}] = \{0\}$.
	\end{enumerate}
	The integer $s$ is called the step of the Carnot group, and $V_{1}$ is called the horizontal layer (or first layer). Consider on $\mathfrak{g}$ an arbitrary inner product that makes the decomposition \eqref{stratification} orthogonal, i.e. for $i \neq j$, $V_i \perp V_j$. This inner product uniquely defines a left-invariant Riemannian metric $\left\langle \cdot,\cdot \right\rangle$ on $G$. We can orthogonally decompose the tangent bundle $TG$ as
	\[
	TG=\mathcal{H}\oplus \mathcal{V}
	\]
	where $\Ho$ is the left invariant sub-bundle which gives $V_1$ at the identity and $\V$ is the left invariant sub-bundle which gives $\oplus_{i \ge 2} V_i$ at the identity. Since $[\mathcal{V},\mathcal{V}] \subset \mathcal{V}$, $\V$ is the vertical bundle of a foliation $\mathcal{F}$ on $G$. Note that $\Ho$ is bracket generating.
    \end{example}

	\subsection{Horizontal Laplacian and heat kernel}
	
	\paragraph{Horizontal Laplacian.}The Riemannian gradient of a function $f$ will be denoted by $\nabla f$ and the horizontal gradient by $\nabla_\Ho f$: it is simply defined as the projection of $\nabla f$ onto $\mathcal{H}$. Similarly, $\nabla_\V$ denotes the vertical gradient. The horizontal Laplacian $\Delta_\Ho$ is the generator of the symmetric $L^2(M,\mu)$-closable bilinear form:
	\[
	\mathcal{E}_{\mathcal{H}} (f,g) 
	=\int_\bM \langle \nabla_\mathcal{H} f , \nabla_\mathcal{H} g \rangle_{\mathcal{H}} \,d\mu, 
	\quad f,g \in C_0^\infty(M),
	\]
	where $\mu$ denotes the Riemannian volume measure on $M$ and  $C_0^\infty(M)$ the space of smooth and compactly supported functions on $M$. Our sign convention is $\int_M f\Delta_\Ho g\,d\mu=-\mathcal E_\Ho(f,g)$.

    \begin{lemma}\label{formula horizontal laplacian}
If $X_1,\cdots,X_n$ is a local orthonormal frame of horizontal vector fields then we locally have
\begin{align}\label{def: horizontal laplace}
		\Delta_{\mathcal{H}}=\sum_{i=1}^n X_i^2 -\sum_{i=1}^n (D_{X_i}X_i)_{\mathcal{H}}-\mathrm H.
	\end{align}
    where $\mathrm H$ is the mean curvature vector field of the leaves.
    \end{lemma}
\begin{proof}
If $f,g\in C_0^\infty(M)$ have a support small enough, then
\[
\mathcal{E}_{\mathcal{H}} (f,g) 
	=\int_\bM \sum_{i=1}^n (X_i f)(X_ig) \,d\mu.
\]
Therefore, we locally have
\[
\Delta_{\mathcal{H}}=-\sum_{i=1}^n X_i^*X_i,
\]
where $X_i^*$ is the formal adjoint of $X_i$ in $L^2(M,\mu)$. Let now $(Z_\ell), 1 \le \ell \le m$ be a local vertical orthonormal frame. Since $\mu$ is the Riemannian volume measure, it is  easy to check that
\[
X_i^*=-X_i+\sum_{j=1}^n \left\langle D_{X_j}X_j,X_i \right\rangle+\sum_{\ell=1}^m\left\langle D_{Z_\ell} Z_\ell ,X_i \right\rangle.
\]
Therefore,  we have
\[
\Delta_{\mathcal{H}}=\sum_{i=1}^n X_i^2 -\sum_{i=1}^n (D_{X_i}X_i)_{\mathcal{H}}-\sum_{\ell=1}^m (D_{Z_\ell} Z_\ell)_\mathcal{H}.
\]
Since the mean curvature vector field of a leaf is given by the trace of the second fundamental form $\Pi(U,V)=(D_U V)_\Ho$, $U,V\in \V$, we have
\[
\mathrm H=\sum_{\ell=1}^m (D_{Z_\ell} Z_\ell)_\mathcal{H}.
\]
The conclusion follows.
\end{proof}

	\paragraph{Heat kernel.}The hypothesis that $\mathcal{H}$ is bracket generating implies that the horizontal Laplacian $\Delta_{\mathcal{H}}$ is locally subelliptic and the completeness assumption on $g$ implies that $\Delta_{\mathcal{H}}$ is furthermore essentially self-adjoint on $C_0^\infty(M)$ and the construction of the heat kernel is then classical, see for instance~\cite{BaudoinEMS2014}. The self-adjoint extension is still denoted by $\Delta_{\mathcal{H}}$. If $\Delta_\mathcal{H}=-\int_0^{+\infty} \lambda dE_\lambda$ denotes the spectral
decomposition of $\Delta_\mathcal{H}$ in $L^2 (M,\mu)$, then by definition, the
heat semigroup $(P_t)_{t \ge 0}$ is given by $P_t= \int_0^{+\infty}
e^{-\lambda t} dE_\lambda$. It is a one-parameter family of bounded operators on
$L^2 (M,\mu)$. Since the closure of the quadratic form $\mathcal{E}_\mathcal{H}$ is a Dirichlet form, $(P_t)_{t \ge 0}$ is a sub-Markov semigroup: it transforms non-negative functions into non-negative functions and satisfies
\begin{equation*}
P_t 1 \le 1.
\end{equation*}
The sub-Markov property and  Riesz-Thorin interpolation classically allows one to construct the semigroup $(P_t)_{t \ge 0}$ in $L^p(M,\mu)$ and for $f \in L^p(M,\mu)$ one has
\begin{equation*}
\|P_tf\|_{L^p(M , \mu)} \le \|f\|_{L^p(M, \mu)},\ \  1\le p\le \infty.
\end{equation*}

By hypoellipticity of $\Delta_\Ho$, there is a smooth function $p(t,x,y)$, $t \in (0,+\infty), x,y \in M$, such that for every $f \in L^p(M,\mu)$, $1 \le p \le \infty$ and $x \in \mathbb{M}$ ,
\[
P_t f (x)=\int_{M} p(t,x,y) f(y) d\mu (y).
\]
The function $p(t,x,y)$ is called the \emph{horizontal} heat kernel associated to $(P_t)_{t \ge 0}$. It satisfies furthermore:
\begin{enumerate}
\item (Symmetry) $p(t,x,y)=p(t,y,x)$;
\item (Chapman-Kolmogorov relation) $p(t+s,x,y)=\int_{M} p(t,x,z)p(s,z,y)d\mu(z)$. 
\end{enumerate}\index{Chapman-Kolmogorov relation}
Moreover, for $f \in L^p(M,\mu)$, $1 < p < \infty$, the function
\[
u (t,x)= P_t f (x), \quad t \ge 0, x\in M.
\]
is the unique solution in $C([0,+\infty);L^p(M,\mu))$ of the Cauchy problem (with the initial condition understood in $L^p$)
\[
\frac{\partial u}{\partial t}= \Delta_{\mathcal H} u,\quad u (0,x)=f(x).
\]

\subsection{Fundamental connection}
	
	The Levi-Civita connection $D$ is, in general, poorly suited to study foliations since the horizontal and vertical bundle might not be $D$-parallel.  There is a more natural connection $\nabla$   that respects the foliation structure, see  \cite{AGAG},  \cite{2025arXiv250913276B}, \cite{Hladky}.  
\begin{proposition}[\cite{Hladky}]
There exists a unique metric connection $\nabla$ on $M$ such that:
\begin{itemize}
\item $\mathcal{H}$ and $\mathcal{V}$ are $\nabla$-parallel, i.e. for every $X \in \mathcal{X}(\hor), Y \in \mathcal{X}(TM), Z \in \mathcal{X} (\V)$,
\begin{equation}
\nabla_Y X \in \mathcal{X}(\hor), \quad \nabla_Y Z \in \mathcal{X}(\V).
\end{equation}
\item The torsion $\Tor^\nabla$ of $\nabla$ satisfies $\Tor^\nabla(\hor,\hor) \subset \V$ and $\Tor^\nabla(\V,\V) \subset \mathcal{H}$.
\item For every $X,Y \in \mathcal{X}(\hor)$, $V,Z \in \mathcal{X} (\V)$, 
\begin{equation}
\langle \Tor^\nabla(X,Z),Y \rangle =\langle \Tor^\nabla(Y,Z),X \rangle, \qquad   \langle \Tor^\nabla(Z,X), V \rangle =\langle \Tor^\nabla(V,X), Z \rangle.
\end{equation}
\end{itemize}
\end{proposition}

\begin{remark}\label{involutivity assumption}
    The connection $\nabla$ is more generally defined and uniquely characterized by the above properties in the context of Riemannian manifolds for which the tangent bundle can orthogonally be split as $TM=\Ho \oplus \V$; the involutivity property $[\V,\V] \subset \V $ is not necessary. The involutivity property is actually equivalent to  $\Tor^\nabla(\V,\V) =0$, see the formula \eqref{formula torsion}.
\end{remark}

The connection $\nabla$ can be expressed in terms of the Levi-Civita connection $D$ by introducing a $(2,1)$ tensor $C$ through the formula:
\begin{equation}
\left\langle C_X Y, Z\right\rangle = \frac{1}{2} (\mathcal{L}_{X_\ver} g)(Y_\hor,Z_\hor) + \frac{1}{2}(\mathcal{L}_{X_\hor}g)(Y_\ver,Z_\ver).
\end{equation}
Notice that the following properties hold:
\begin{equation}
C_\ver \ver = 0,\qquad C_{\ver} \hor \subseteq \hor, \qquad C_\hor \hor = 0,\qquad C_{\hor} \ver \subseteq \ver.
\end{equation}

The connection $\nabla$ can then be expressed in terms of the Levi-Civita one  by
\[
	\nabla_X Y =
	\begin{cases}
		( D_X Y)_{\mathcal{H}} , &X,Y \in \mathfrak{X}(\mathcal{H}), \\
		[X,Y]_{\mathcal{H}}+C_X Y,  &X \in \mathfrak{X}(\mathcal{V}),\ Y \in \mathfrak{X}(\mathcal{H}), \\
		[X,Y]_{\mathcal{V}}+C_X Y,  &X \in \mathfrak{X}(\mathcal{H}),\ Y \in \mathfrak{X}(\mathcal{V}), \\
		(D_X Y)_{\mathcal{V}} , &X,Y \in \mathfrak{X}(\mathcal{V}).
	\end{cases}
	\]
    
and its torsion is given by
\begin{equation}\label{formula torsion}
\Tor^\nabla(X,Y) = \begin{cases}
- [X,Y]_\V & X,Y \in \mathfrak{X}(\mathcal{H}), \\
C_X Y - C_Y X &  X \in \mathfrak{X}(\mathcal{H}), Y \in \mathfrak{X}(\mathcal{V}), \\
0 & X,Y \in \mathfrak{X}(\mathcal{V}).
\end{cases}
\end{equation}

For $Z \in \mathfrak{X}(TM)$, there is a  unique skew-symmetric endomorphism  $J_Z:T_xM \to T_xM$ such that for all vector fields $X$ and $Y$,
\begin{align}\label{Jmap}
\left\langle J_Z X,Y\right\rangle= \left\langle Z,\Tor^\nabla (X,Y) \right\rangle.
\end{align}
With this notation, one can easily check that the relation between the Levi-Civita connection $D$  and the connection $\nabla$ is given by the formula

\begin{equation}\label{Levi-Civita}
\nabla_X Y=D_X Y+\frac{1}{2} \Tor^\nabla (X,Y)-\frac{1}{2} J_XY-\frac{1}{2} J_Y X, \quad X,Y \in \mathfrak{X}(TM).
\end{equation}

\begin{remark}\label{torsion characterization}
    From the torsion formula one can see that:
    \begin{itemize}
        \item The metric $g$ is bundle-like i.e. $\mathcal{L}_Z g (X,X)=0$ for  $Z \in \mathcal{X} (\V), X \in \mathcal{X}(\Ho)$ if and only if $T(\Ho,\V) \subset \V$;
        \item The leaves are totally geodesic i.e. $ D_U V \in \mathcal{X} (\V)$ for $U,V \in \mathcal{X} (\V)$ if and only if $\mathcal{L}_X g (Z,Z)=0$ for  $Z \in \mathcal{X} (\V), X \in \mathcal{X}(\Ho)$ if and only if $T(\Ho,\V) \subset \Ho$.
    \end{itemize}
\end{remark}

\begin{example}[Contact manifold]
Let $(M,\theta,g)$ be a contact Riemannian manifold as in Example \ref{Example contact}. In that case, one can check that
\[
\nabla_XY=D_XY-\theta(X)JY-\theta(Y)D_X \xi +[(D_X\theta)Y]\xi .
\] 
Therefore $\nabla$ coincides with the Tanno's connection introduced in \cite{Tanno}.
\end{example}

\begin{example}[Carnot groups]
Consider the foliation on a Carnot group from Example \ref{Carnot example}. For a left invariant vector field $X$, denote $\mathrm{ad}_X$ the map $\mathrm{ad}_X (Y)=[X,Y]$ and $\mathrm{ad}^*$ its adjoint. It follows from Koszul's formula that for left invariant vector fields
\[
\nabla_X Y =
\begin{cases}
0, &X,Y \in \mathfrak{X}(\mathcal{H}), \\
  0,  &X \in \mathfrak{X}(\mathcal{V}),\ Y \in \mathfrak{X}(\mathcal{H}), \\
 \frac{1}{2} \mathrm{ad}_X Y-\frac{1}{2}(\mathrm{ad}_X^*Y)_\V,  &X \in \mathfrak{X}(\mathcal{H}),\ Y \in \mathfrak{X}(\mathcal{V}), \\
 \frac{1}{2} \mathrm{ad}_X Y-\frac{1}{2}(\mathrm{ad}^*_Y X)_\V-\frac{1}{2}(\mathrm{ad}^*_X Y)_\V, &X,Y \in \mathfrak{X}(\mathcal{V}).
\end{cases}
\]

\end{example}

The horizontal Laplacian $ \overset{\circ}\Delta_\Ho$ of the connection $\nabla$ is defined as the trace of the $\nabla$-Hessian in the horizontal directions. It is therefore  given in a local horizontal orthonormal frame $X_i$ by
        \[
        \overset{\circ}\Delta_\Ho=\sum_{i=1}^n \nabla_{X_i}\nabla_{X_i} - \nabla_{X_i}X_i.
        \]
Notice that from Lemma \ref{formula horizontal laplacian} and the definition of $\nabla$ we therefore have
\begin{align}\label{split horizontal laplacian}
\Delta_\Ho=\overset{\circ}\Delta_\Ho-\mathrm{H}.
\end{align}

	\section{Bochner's formulas and curvature dimension inequalities}

    Our first main goal in this section is to prove Bochner's type formulas for the horizontal Laplacian. The first formula is in horizontal directions and the second one in vertical directions. The key point is to express the quantities 
    \begin{align*}
			\frac{1}{2}\Delta_{\mathcal{H}} | \nabla_\Ho f|^2-\langle \nabla_{\mathcal{H}} f , \nabla_{\mathcal{H}}\Delta_{\mathcal{H}} f \rangle
		\end{align*}
		and
		\begin{align*}
			\frac{1}{2}\Delta_{\mathcal{H}} | \nabla_\V f|^2-\langle \nabla_{\mathcal{V}} f , \nabla_{\mathcal{V}}\Delta_{\mathcal{H}} f \rangle
		\end{align*}
        from tensors related to the connection $\nabla$. Those formulas generalize both the formulas obtained in \cite{BG17} for sub-Riemannian manifolds with transverse symmetries and \cite{BW14} for contact manifolds.
        
        We will use the following notations. First, recalling that $\mathrm H$ denotes the mean curvature vector field, we define for $U, V \in \mathcal{X}(TM)$
        \[
        \nabla^{\mathrm{sym}} \mathrm H(U,V)=\frac{1}{2} \left\langle \nabla_U \mathrm H,V \right\rangle+\frac{1}{2} \left\langle \nabla_V \mathrm H,U \right\rangle.
        \]
        For a smooth function $f$ and $U, V \in \mathcal{X}(TM)$ we define
        \begin{align*}
        \mathrm{Hess}^{\nabla, \mathrm{sym}} f(U,V)&=\frac{1}{2} (\mathrm{Hess}^{\nabla } f(U,V)+\mathrm{Hess}^{\nabla } f(V,U)) \\
         &=\frac{1}{2} \left(UV+VU -\nabla_U V-\nabla_VU \right)f
        \end{align*}
        and
        \begin{align*}
        \mathrm{Hess}_\Ho^{\nabla, \mathrm{sym}} f(U,V)=\mathrm{Hess}^{\nabla, \mathrm{sym}} f(U_\Ho,V_\Ho).
        \end{align*}

        For the following notations the $X_i$'s below form an arbitrary orthonormal local frame of horizontal vectors and $U, V$ are arbitrary vectors  in $\mathcal{X}(TM)$.
	
	\begin{itemize}
		\item The horizontal Ricci curvature of the connection $\nabla$ is defined as the $(0,2)$ tensor
		\[
		\mathrm{Ric}^\nabla_\Ho(U,V)=\sum_{i=1}^n  \left \langle \mathrm{Riem}^\nabla(U,X_i)X_i, V \right.\rangle
		\]
		\item The horizontal divergence of the torsion is defined as the $(1,1)$ tensor
		\[
		\delta_\Ho \Tor^\nabla (U)= \sum_{i=1}^n (\nabla_{X_i} \Tor^\nabla)(X_i,U).
		\]
		\item
		\[
		(\Tor^\nabla(U) ,\Tor^\nabla(V))_\Ho=\sum_{i=1}^n \left\langle \Tor^\nabla (U,X_i), \Tor^\nabla (V,X_i)\right\rangle
		\]
        \item 
        \[
       \tau (U,V)=\sum_{i=1}^n \langle \Tor^\nabla(X_i, \Tor^{\nabla}(X_i, U) ), V  \rangle
        \]
        \item 
        \[
        \iota(U)=\sum_{i=1}^n \langle \Tor^\nabla(U,X_i),X_i \rangle,\qquad
        \beta(U)=\langle\mathrm H,U_\Ho\rangle+\iota(U_\V).
        \]
		\item 
		\[
		(J_U ,J_V)_\Ho=\sum_{i=1}^n \left\langle J_U X_i , J_V X_i \right\rangle_\Ho.
		\]
	\end{itemize}
	
	\subsection{Horizontal and vertical Bochner's formulas for the horizontal Laplacian}

    Using the notations introduced above, the Bochner's formulas write as follows.
    
	\begin{theorem}\label{Bochner}
		Let $f \in C^\infty (M)$ and  $X_1,\cdots,X_n$ be a local orthonormal frame of horizontal vector fields. We have
		\begin{align*}
\frac12\Delta_\Ho|\nabla_\Ho f|^2
={}&\langle\nabla_\Ho f,\nabla_\Ho\Delta_\Ho f\rangle
+2\sum_i\langle\nabla_{X_i}\nabla_\V f,\Tor^\nabla(\nabla_\Ho f,X_i)\rangle\\
&+|\mathrm{Hess}_\Ho^{\nabla,\mathrm{sym}}f|^2
-\langle\nabla_\V f,\delta_\Ho\Tor^\nabla(\nabla_\Ho f)\rangle\\
&+\mathrm{Ric}_\Ho^\nabla(\nabla_\Ho f,\nabla_\Ho f)
-(\Tor^\nabla(\nabla_\Ho f),\Tor^\nabla(\nabla_\V f))_\Ho\\
&+\frac14(J_{\nabla_\V f},J_{\nabla_\V f})_\Ho
+\nabla^{\mathrm{sym}}\mathrm H(\nabla_\Ho f,\nabla_\Ho f)\\
&+\langle\Tor^\nabla(\mathrm H,\nabla_\Ho f),\nabla f\rangle
-\tau(\nabla_\Ho f,\nabla_\Ho f).
\end{align*}
		and
		\begin{align*}
			\frac{1}{2}\Delta_{\mathcal{H}} | \nabla_\V f|^2=&\langle \nabla_{\mathcal{V}} f , \nabla_{\mathcal{V}}\Delta_{\mathcal{H}} f \rangle+2\sum_i\langle \nabla_{X_i} \nabla f, \mathrm{Tor}^\nabla(\nabla_{\mathcal{V}}f, X_i) \rangle+\mid \nabla_{\mathcal{H}} \nabla_{\mathcal{V}} f \mid^2\\ 
			&-\langle \nabla f, \delta_\Ho \mathrm{Tor}^\nabla ( \nabla_\V f ) \rangle-\sum_{i=1}^n|\Tor^\nabla(\nabla_\V f,X_i)_\V|^2 +\mathrm{Ric}^\nabla_\Ho(\nabla_\V f,\nabla_\Ho f)\\
            &+2\nabla^{\mathrm{sym}}\mathrm H(\nabla_\Ho f,\nabla_\V f)+\left\langle \mathrm{Tor}^\nabla (\mathrm H, \nabla_\V f), \nabla f \right\rangle-\tau(\nabla_\V f,\nabla_\Ho f).
		\end{align*}
	\end{theorem}
	
	The proof is rather long and partly inspired by \cite{Hladky}. We start  with four preliminary lemmas. In what follows, $f$ is a fixed function in $C^\infty(M)$. The first lemma symmetrizes the Hilbert-Schmidt norm of the horizontal Hessian for the connection $\nabla$.

	\begin{lemma}\label{symmetrized-hessian}
		\begin{align*}
			|\nabla_\Ho \nabla_\Ho f|^2 =| \mathrm{Hess}_\Ho^{\nabla, \mathrm{sym}} f|^2 +\frac{1}{4}( J_{\nabla_\V f} , J_{\nabla_\V f} )_\Ho
		\end{align*}
	\end{lemma}
	
	\begin{proof}
		If $X_1,\cdots, X_n$ is a local horizontal orthonormal frame then
		\begin{align*}
			|\nabla_\Ho \nabla_\Ho f|^2 &=\sum_{i=1}^n |\nabla_{X_i} \nabla_\Ho f|^2 \\
			&=\sum_{i,j=1}^n \langle \nabla_{X_i} \nabla_\Ho f, X_j \rangle^2 \\
			&=\sum_{i,j=1}^n \left( \frac{1}{2} \langle \nabla_{X_i} \nabla_\Ho f, X_j \rangle+\frac{1}{2} \langle \nabla_{X_i} \nabla_\Ho f, X_j \rangle \right)^2  \\
			&=\sum_{i,j=1}^n \left( \frac{1}{2} \langle \nabla_{X_i} \nabla_\Ho f, X_j \rangle+\frac{1}{2} \langle \nabla_{X_j} \nabla_\Ho f, X_i \rangle+\frac{1}{2} \langle \Tor^\nabla (X_j,X_i) ,\nabla f \rangle \right)^2 \\
			&=\sum_{i,j=1}^n \left( \frac{\langle \nabla_{X_i} \nabla_\Ho f, X_j \rangle+\langle \nabla_{X_j} \nabla_\Ho f, X_i \rangle}{2} \right)^2+\frac{1}{4} \sum_{i,j=1}^n \langle \Tor^\nabla (X_j,X_i) ,\nabla f \rangle ^2.
		\end{align*}
        Note that the mixed terms vanish because of the anti-symmetry of the torsion tensor. Since $\Tor^\nabla(\hor,\hor) \subset \V$, we then have
		\begin{align*}
			\sum_{i,j=1}^n \langle \Tor^\nabla (X_j,X_i) ,\nabla f \rangle ^2&=\sum_{i,j=1}^n \langle \Tor^\nabla (X_j,X_i) ,\nabla_\V f \rangle ^2 \\
			&=\sum_{i,j=1}^n \langle X_i ,J_{\nabla_\V f} X_j \rangle ^2 \\
            &=\sum_{j=1}^n |J_{\nabla_\V f} X_j|_\Ho ^2 \\
			&=( J_{\nabla_\V f} , J_{\nabla_\V f} )_\Ho.
		\end{align*}
	\end{proof}
	The second lemma deals with Ricci type commutation identities related to the connection $\nabla$.
    
	\begin{lemma}\label{partial hessian} 
		If $X_1,\cdots, X_n$ is a local horizontal orthonormal frame then
		\[
		\sum_{i=1}^n \langle[\nabla_{\nabla_\mathcal{H}f} \nabla_{X_i}-\nabla_{\nabla_{\nabla_\mathcal{H}f} X_i}] \nabla f, X_i \rangle=\langle \nabla_{\mathcal{H}} f , \nabla_{\mathcal{H}}\overset{\circ}\Delta_{\mathcal{H}} f \rangle
		\]
		and
		\[
		\sum_{i=1}^n \langle[\nabla_{\nabla_\mathcal{V}f} \nabla_{X_i}-\nabla_{\nabla_{\nabla_\mathcal{V}f} X_i}] \nabla f, X_i \rangle=\langle \nabla_{\mathcal{V}} f , \nabla_{\mathcal{V}}\overset{\circ}\Delta_{\mathcal{H}} f \rangle.
		\]
	\end{lemma}
	
	\begin{proof}
		We have
		\begin{align*}
			\overset{\circ}\Delta_{\mathcal{H}} f&=\sum_{i=1}^n \mathrm{Hess}^\nabla f (X_i,X_i) \\
			&=\sum_{i=1}^n \langle \nabla_{X_i} \nabla f, X_i \rangle.
		\end{align*}
		Therefore we have
		\begin{align*}
			\langle \nabla_\Ho f ,\nabla_\Ho \overset{\circ}\Delta_{\mathcal{H}} f \rangle&=\sum_{i,j=1}^n (X_j \langle \nabla_{X_i} \nabla f, X_i \rangle) X_j f \\
			&=\sum_{i,j=1}^n ( \langle \nabla_{X_j} \nabla_{X_i} \nabla f, X_i \rangle+ \langle \nabla_{X_i} \nabla f, \nabla_{X_j} X_i \rangle) X_j f \\
			&=\sum_{i=1}^n \langle \nabla_{\nabla_\Ho f } \nabla_{X_i} \nabla f, X_i \rangle+ \sum_{i,j=1}^n\langle \nabla_{X_i} \nabla f, \nabla_{X_j} X_i \rangle) X_j f 
		\end{align*}
		Using that $\langle \nabla_{X_j} X_i,X_k \rangle=-\langle \nabla_{X_j} X_k,X_i \rangle$ we now compute
		\begin{align*}
			 \sum_{i,j=1}^n  \langle \nabla_{X_i} \nabla f, \nabla_{X_j} X_i \rangle X_j f
			&=\sum_{i,j,k=1}^n  \langle \nabla_{X_i} \nabla f, X_k \rangle \langle \nabla_{X_j} X_i,X_k \rangle X_j f \\
			&=-\sum_{i,j,k=1}^n  \langle \nabla_{X_i} \nabla f, X_k \rangle \langle \nabla_{X_j} X_k,X_i \rangle X_j f \\
			&=-\sum_{j,k=1}^n  \langle \nabla_{\nabla_{X_j} X_k} \nabla f, X_k \rangle  X_j f \\
			&=-\sum_{k=1}^n  \langle \nabla_{\nabla_{\nabla_\Ho f} X_k} \nabla f, X_k \rangle . 
		\end{align*}
		We conclude
		\[
		\sum_{i=1}^n \langle[\nabla_{\nabla_\mathcal{H}f} \nabla_{X_i}-\nabla_{\nabla_{\nabla_\mathcal{H}f} X_i}] \nabla f, X_i \rangle=\langle \nabla_{\mathcal{H}} f , \nabla_{\mathcal{H}}\overset{\circ}\Delta_{\mathcal{H}} f \rangle.
		\]
		The second computation proceeds almost in the same way. We first have
		\begin{align*}
			\langle \nabla_\V f ,\nabla_\V \overset{\circ}\Delta_{\mathcal{H}} f \rangle&=\sum_{i=1}^n\left(\langle \nabla_{\nabla_\V f } \nabla_{X_i} \nabla f, X_i \rangle+ \langle \nabla_{X_i} \nabla f, \nabla_{\nabla_\V f } X_i \rangle\right).
		\end{align*}
		and then
		\begin{align*}
			\sum_{i=1}^n \langle \nabla_{X_i} \nabla f, \nabla_{\nabla_\V f } X_i \rangle 
			&=\sum_{i,k=1}^n\sum_{\ell=1}^m  \langle \nabla_{X_i} \nabla f, X_k \rangle \langle \nabla_{Z_\ell} X_i,X_k \rangle Z_\ell f \\
			&=-\sum_{i,k=1}^n\sum_{\ell=1}^m  \langle \nabla_{X_i} \nabla f, X_k \rangle \langle \nabla_{Z_\ell} X_k,X_i \rangle Z_\ell f \\
			&=-\sum_{k=1}^n  \langle \nabla_{\nabla_{\nabla_\V f} X_k} \nabla f, X_k \rangle . 
		\end{align*}
		
	\end{proof}
	
	For the remainder of the proof, define 
	\begin{align*}
		u_1=\frac{1}{2}\abs{\nabla_{\mathcal{H}} f}^2,\; u_2=\frac{1}{2}\abs{\nabla_{\mathcal{V}} f}^2.
	\end{align*}
	\begin{lemma}\label{horizontal-part-grad}
		We have
		\begin{align*}
			\nabla_{\mathcal{H}}u_1&=\nabla_{\nabla_\mathcal{H}f} (\nabla_\mathcal{H}f)+(J_{\nabla f}\nabla_\mathcal{H}f)_{\mathcal{H}},\\
			\nabla_{\mathcal{H}}u_2&=\nabla_{\nabla_\mathcal{V}f} (\nabla_\mathcal{H}f)+(J_{\nabla  f}\nabla_\mathcal{V}f)_{\mathcal{H}}.
		\end{align*}
	\end{lemma}
	\begin{proof}
		Since $\nabla$ is a metric connection, we have for any $X\in \mathfrak{X}(\mathcal{H})$,
		\begin{align*}
			\frac{1}{2}X\langle \nabla_\mathcal{H}f, \nabla_\mathcal{H}f\rangle&= \langle \nabla_X \nabla_\mathcal{H}f ,\nabla_\mathcal{H}f\rangle=\langle \nabla_X \nabla f ,\nabla_\mathcal{H}f\rangle\\
			&=\mathrm{Hess}^{\nabla}f(X, \nabla_{\mathcal{H}} f)=\mathrm{Hess}^{\nabla}f( \nabla_{\mathcal{H}} f, X)-\langle \mathrm{Tor}^{\nabla}(X, \nabla_{\mathcal{H}} f),\nabla f\rangle \\
			&=\langle \nabla_{\nabla_\mathcal{H}f} \nabla f , X\rangle+\langle \mathrm{Tor}^{\nabla}( \nabla_{\mathcal{H}} f, X),\nabla f\rangle \\
			&=\langle \nabla_{\nabla_\mathcal{H}f} \nabla_\Ho f , X\rangle+\langle \mathrm{Tor}^{\nabla}( \nabla_{\mathcal{H}} f, X),\nabla f\rangle
		\end{align*}
		Recall now the definition of the $J$ tensor
		\[ \langle J_ZX, Y \rangle=\langle Z, \Tor^\nabla (X,Y) \rangle.   \]
		We thus get
		\begin{align*}
			\langle \nabla_{\mathcal{H}} u_1, X \rangle=\frac{1}{2}X\langle \nabla_\mathcal{H}f, \nabla_\mathcal{H}f\rangle=\langle \nabla_{\nabla_\mathcal{H}f} \nabla_\mathcal{H} f , X\rangle+\langle (J_{\nabla f}\nabla_\mathcal{H}f)_{\mathcal{H}} , X \rangle
		\end{align*}
		Since this holds for every $X \in \mathfrak{X}(\Ho)$, this implies
		\begin{align*}
			\nabla_{\mathcal{H}}u_1=\nabla_{\nabla_\mathcal{H}f} \nabla_\mathcal{H}f+(J_{\nabla f}\nabla_\mathcal{H}f)_{\mathcal{H}}.
		\end{align*}
		Similarly, we have for any $X\in \mathfrak{X} (\mathcal{H})$,
		\begin{align*}
			\frac{1}{2}X\langle \nabla_\mathcal{V}f, \nabla_\mathcal{V}f\rangle&= \langle \nabla_X \nabla_\mathcal{V}f ,\nabla_\mathcal{V}f\rangle =\langle \nabla_X \nabla f ,\nabla_\mathcal{V}f\rangle\\
			&=\mathrm{Hess}^{\nabla}f(X, \nabla_{\mathcal{V}} f)=\mathrm{Hess}^{\nabla}f( \nabla_{\mathcal{V}} f, X)-\langle \mathrm{Tor}^{\nabla}(X, \nabla_{\mathcal{V}} f), \nabla f\rangle \\
			&=\langle \nabla_{\nabla_\mathcal{V}f} \nabla f , X\rangle+\langle \mathrm{Tor}^{\nabla}( \nabla_{\mathcal{V}} f, X),\nabla f\rangle \\
			&=\langle \nabla_{\nabla_\mathcal{V}f} \nabla f , X\rangle+\langle J_{\nabla f}\nabla_\mathcal{V}f, X \rangle.
		\end{align*}		
	\end{proof}
	The next lemma deals with the contribution of the mean curvature vector.
    \begin{lemma}\label{mean curvature part}
\begin{align*}
			\frac{1}{2}\mathrm H | \nabla_\Ho f|^2-\langle \nabla_{\mathcal{H}} f , \nabla_{\mathcal{H}} \mathrm H f \rangle=- \nabla^{\mathrm{sym}}\mathrm H(\nabla_\Ho f,\nabla_\Ho f)-\left\langle \mathrm{Tor}^\nabla (\mathrm H, \nabla_\Ho f), \nabla f \right\rangle 
		\end{align*}
		and
		\begin{align*}
			\frac{1}{2}\mathrm H | \nabla_\V f|^2-\langle \nabla_{\mathcal{V}} f , \nabla_{\mathcal{V}}\mathrm H f \rangle=-2\nabla^{\mathrm{sym}}\mathrm H(\nabla_\V f,\nabla_\Ho f)-\left\langle \mathrm{Tor}^\nabla (\mathrm H, \nabla_\V f), \nabla f \right\rangle.
		\end{align*}
    \end{lemma}

    \begin{proof}
Let $X_1,\cdots,X_n$ be a local  horizontal orthonormal frame and $Z_1,\cdots,Z_m$ be a local vertical orthonormal frame. Since $\mathrm H$ is horizontal we have
\begin{align*}
& \frac{1}{2}\mathrm H | \nabla_\Ho f|^2-\langle \nabla_{\mathcal{H}} f , \nabla_{\mathcal{H}}\mathrm H f \rangle \\
=&\sum_{i=1}^n (\mathrm HX_i f)(X_i f)-\sum_{i=1}^n (X_i \mathrm Hf)(X_i f) \\
=&\sum_{i=1}^n ([\mathrm H,X_i]f)(X_i f) \\
=&\sum_{i,j=1}^n \left\langle [\mathrm H,X_i],X_j \right\rangle (X_j f)(X_i f)+\sum_{i=1}^n\sum_{\ell=1}^m  \left\langle [\mathrm H,X_i],Z_\ell \right\rangle (Z_\ell f)(X_i f)\\
=&-\sum_{i,j=1}^n \left\langle \nabla_{X_i} \mathrm H,X_j \right\rangle (X_j f)(X_i f)+\sum_{i,j=1}^n \left\langle \nabla_{\mathrm H}X_i ,X_j \right\rangle (X_j f)(X_i f)\\
&-\left\langle \mathrm{Tor}^\nabla (\mathrm H, \nabla_\Ho f), \nabla_\Ho f \right\rangle-\left\langle \mathrm{Tor}^\nabla (\mathrm H, \nabla_\Ho f), \nabla_\V f \right\rangle \\
\end{align*}
Since $\left\langle \nabla_{\mathrm H}X_i ,X_j \right\rangle=-\left\langle \nabla_{\mathrm H}X_j ,X_i \right\rangle$, we can write

\begin{align*}
& \frac{1}{2}\mathrm H | \nabla_\Ho f|^2-\langle \nabla_{\mathcal{H}} f , \nabla_{\mathcal{H}}\mathrm H f \rangle \\
=&-\sum_{i,j=1}^n \left\langle \nabla_{X_i} \mathrm H,X_j \right\rangle (X_j f)(X_i f)-\left\langle \mathrm{Tor}^\nabla (\mathrm H, \nabla_\Ho f), \nabla_\Ho f \right\rangle-\left\langle \mathrm{Tor}^\nabla (\mathrm H, \nabla_\Ho f), \nabla_\V f \right\rangle \\
=&- \nabla^{\mathrm{sym}}\mathrm H(\nabla_\Ho f,\nabla_\Ho f)-\left\langle \mathrm{Tor}^\nabla (\mathrm H, \nabla_\Ho f), \nabla f \right\rangle.    
\end{align*}

The computation for $\frac{1}{2}\mathrm H | \nabla_\V f|^2-\langle \nabla_{\mathcal{V}} f , \nabla_{\mathcal{V}}\mathrm H f \rangle$ follows the same pattern:
\begin{align*}
& \frac{1}{2}\mathrm H | \nabla_\V f|^2-\langle \nabla_{\mathcal{V}} f , \nabla_{\mathcal{V}}\mathrm H f \rangle \\
=&\sum_{\ell=1}^m ([\mathrm H,Z_\ell]f)(Z_\ell f) \\
=&\sum_{\ell=1}^m \sum_{j=1}^n \left\langle [\mathrm H,Z_\ell],X_j \right\rangle (Z_\ell f) (X_j f)+\sum_{k,\ell=1}^m  \left\langle [\mathrm H,Z_\ell],Z_k \right\rangle (Z_\ell f)(Z_k f)\\
=&-\sum_{\ell=1}^m \sum_{j=1}^n \left\langle \nabla_{Z_\ell} \mathrm H,X_j \right\rangle (X_j f)(Z_\ell f)-\sum_{\ell=1}^m \sum_{j=1}^n \left\langle \Tor^\nabla( \mathrm H,Z_\ell),X_j \right\rangle (X_j f)(Z_\ell f)  \\
 &-\left\langle \mathrm{Tor}^\nabla (\mathrm H, \nabla_\V f), \nabla_\V f \right\rangle- \nabla^{\mathrm{sym}}\mathrm H(\nabla_\V f,\nabla_\V f) 
\end{align*}
Finally, note that since $\mathrm H$ is horizontal, $\left\langle \nabla_{X_j} \mathrm H, Z_\ell\right\rangle =\nabla^{\mathrm{sym}}\mathrm H(\nabla_\V f,\nabla_\V f)=0$. We conclude
\begin{align*}
& \frac{1}{2}\mathrm H | \nabla_\V f|^2-\langle \nabla_{\mathcal{V}} f , \nabla_{\mathcal{V}}\mathrm H f \rangle \\
   =&-2\nabla^{\mathrm{sym}}\mathrm H(\nabla_\V f,\nabla_\Ho f)-\left\langle \mathrm{Tor}^\nabla (\mathrm H, \nabla_\V f), \nabla f \right\rangle.
\end{align*}
    \end{proof}
	We can now proceed to the proof of Theorem \ref{Bochner}.

	\begin{proof}[Proof of Theorem \ref{Bochner}]
		
		Plug in our result from Lemma \ref{horizontal-part-grad}, we get that for every $X \in \mathfrak{X}(\Ho)$,
		\begin{align*}
			&\langle \nabla_X \nabla_{\mathcal{H}}u_1 ,X \rangle\\
			=&\langle \nabla_X \nabla_{\nabla_\mathcal{H}f} \nabla_{\mathcal{H}}f, X \rangle+\langle \nabla_X (J_{\nabla f}\nabla_\mathcal{H}f)_{\mathcal{H}}, X \rangle\\
			=&\langle \nabla_X \nabla_{\nabla_\mathcal{H}f} \nabla f, X \rangle+\langle \nabla_X (J_{\nabla f}\nabla_\mathcal{H}f), X \rangle\\
			=&\langle \nabla_{\nabla_\mathcal{H}f} \nabla_X  \nabla f, X \rangle+\langle \mathrm{Riem}^{\nabla}(X, \nabla_{\mathcal{H}} f) \nabla f, X\rangle\\
			&+{\langle \nabla_{[X, \nabla_\mathcal{H}f]}  \nabla f, X \rangle}+\langle \nabla_X (J_{\nabla f}\nabla_\mathcal{H}f), X \rangle\\
			=&\langle \nabla_{\nabla_\mathcal{H}f} \nabla_X  \nabla f, X \rangle+\langle \mathrm{Riem}^{\nabla}(X, \nabla_{\mathcal{H}} f) \nabla_\Ho f, X\rangle\\
			&+{\langle \nabla_{\nabla_X \nabla_\mathcal{H}f-\nabla_{\nabla_\mathcal{H}f} X-\mathrm{Tor}^{\nabla}(X, \nabla_\mathcal{H}f) } \nabla f, X \rangle}\\
			+&\langle \nabla_X (J_{\nabla f}\nabla_\mathcal{H}f), X \rangle\\
			=&\langle[\nabla_{\nabla_\mathcal{H}f} \nabla_X-\nabla_{\nabla_{\nabla_\mathcal{H}f} X}] \nabla f, X \rangle+ \langle\mathrm{Riem}^{\nabla}(X, \nabla_{\mathcal{H}} f) \nabla_\Ho f, X\rangle\\
			&+\langle \nabla_{\nabla_X \nabla_\mathcal{H}f-\Tor^{\nabla}(X, \nabla_\mathcal{H}f) } \nabla f, X \rangle\\
			+&\langle \nabla_X (J_{\nabla f}\nabla_\mathcal{H}f), X \rangle.
		\end{align*}
		Recall that we have for all $A,B\in \mathfrak{X}(TM)$,
		\[ \langle \nabla_A \nabla f, B \rangle=\langle \nabla_B \nabla f, A \rangle+\langle \Tor^\nabla(B,A), \nabla f \rangle.   \]
		Thus
		\begin{align*}
			&\langle \nabla_{\nabla_X \nabla_\mathcal{H}f} \nabla f, X \rangle-\langle \nabla_{\Tor^{\nabla}(X, \nabla_\mathcal{H}f) } \nabla f, X \rangle\\
			=&\langle \nabla_{X} \nabla f, \nabla_X \nabla_\mathcal{H}f  \rangle+\langle \Tor^{\nabla}(X, \nabla_X \nabla_\mathcal{H}f ), \nabla f \rangle\\
			-&\langle \nabla_X\nabla f , \Tor^\nabla(X, \nabla_{\mathcal{H}}f) \rangle-\langle \Tor^\nabla(X, \Tor^{\nabla}(X, \nabla_\mathcal{H}f) ), \nabla f \rangle.
		\end{align*}
		Moreover, since $\nabla$ is metric-compatible, we have
		\begin{align*}
			&\langle \nabla_X (J_{\nabla f}\nabla_\mathcal{H}f), X \rangle\\
			=&X\langle J_{\nabla f}\nabla_\mathcal{H}f, X \rangle-\langle J_{\nabla f}\nabla_\mathcal{H}f, \nabla_X X \rangle  \\
			=&X\langle \nabla f, \Tor^\nabla(\nabla_\mathcal{H}f, X) \rangle-\langle   \nabla f, \Tor^\nabla(\nabla_{\mathcal{H}}f, \nabla_X X)\rangle\\
			=&\langle \nabla_X \nabla f, \Tor^\nabla(\nabla_{\mathcal{H}}f, X) \rangle+\langle \nabla f, \nabla_X (\Tor^\nabla(\nabla_{\mathcal{H}}f, X)) \rangle-\langle \nabla f,  \Tor^\nabla(\nabla_{\mathcal{H}}f, \nabla_X X) \rangle \\
			=&\langle \nabla_X \nabla f, \Tor^\nabla(\nabla_{\mathcal{H}}f, X) \rangle+\langle \nabla f, (\nabla_X \Tor^\nabla)(\nabla_{\mathcal{H}}f, X) \rangle+\langle \nabla f,  \Tor^\nabla(\nabla_X \nabla_{\mathcal{H}}f,  X) \rangle.
		\end{align*}
		
		Therefore
		\begin{align*}
			& \langle \nabla_{\nabla_X \nabla_\mathcal{H}f-\Tor^{\nabla}(X, \nabla_\mathcal{H}f) } \nabla f, X \rangle\
			+\langle \nabla_X (J_{\nabla f}\nabla_\mathcal{H}f), X \rangle \\
			=&\langle \nabla_{X} \nabla f, \nabla_X \nabla_\mathcal{H}f  \rangle+2\langle \nabla_X \nabla f, \Tor^\nabla(\nabla_{\mathcal{H}}f, X) \rangle+\langle \nabla f, (\nabla_X \Tor^\nabla)(\nabla_{\mathcal{H}}f, X) \rangle \\
			&-\langle \Tor^\nabla(X, \Tor^{\nabla}(X, \nabla_\mathcal{H}f) ), \nabla f \rangle \\
			=&|\nabla_X \nabla_\mathcal{H}f  |^2+2\langle \nabla_X \nabla_\V f, \Tor^\nabla(\nabla_{\mathcal{H}}f, X) \rangle+\langle \nabla_\V f, (\nabla_X \Tor^\nabla)(\nabla_{\mathcal{H}}f, X) \rangle \\
			&-\langle \Tor^{\nabla}(X, \nabla_\mathcal{V}f) ,\Tor^{\nabla}(X, \nabla_\mathcal{H}f)\rangle-\langle \Tor^\nabla(X, \Tor^{\nabla}(X, \nabla_\mathcal{H}f) ), \nabla_\Ho f \rangle.
		\end{align*}
		In the last line, we used the fact that the torsion of the connection $\nabla$ satisfies for $U,V \in\mathfrak{X}(\V)$ and $X,Y \in \mathfrak{X}(\Ho)$
		\[
		\langle \Tor^\nabla (U,X),V \rangle=\langle \Tor^\nabla (V,X),U \rangle, \qquad \Tor^\nabla (X,Y) \in \mathcal{X}(\V).
		\]
		We put everything together and get		
		\begin{align*}
			\langle \nabla_X \nabla_{\mathcal{H}}u_1 ,X \rangle=&\langle[\nabla_{\nabla_\mathcal{H}f} \nabla_X-\nabla_{\nabla_{\nabla_\mathcal{H}f} X}] \nabla f, X \rangle +|\nabla_X \nabla_\mathcal{H}f  |^2\\
			&+\langle\mathrm{Riem}^{\nabla}(X, \nabla_{\mathcal{H}} f) \nabla_\Ho f, X\rangle \\
			&+2\langle \nabla_X \nabla_\V f, \Tor^\nabla(\nabla_{\mathcal{H}}f, X) \rangle+\langle \nabla_\V f, (\nabla_X \Tor^\nabla)(\nabla_{\mathcal{H}}f, X) \rangle \\
			& -\langle \Tor^{\nabla}(X, \nabla_\mathcal{V}f) ,\Tor^{\nabla}(X, \nabla_\mathcal{H}f)\rangle-\langle \Tor^\nabla(X, \Tor^{\nabla}(X, \nabla_\mathcal{H}f) ), \nabla_\Ho f \rangle.
		\end{align*}
		Finally, let $X$ range over a horizontal frame, and we get the result  from Lemmas \ref{mean curvature part}, \ref{symmetrized-hessian} and \ref{partial hessian} after summing up.

        We now turn to the second Bochner formula. The same computation gives, for
$v=\nabla_\V f$ and every horizontal vector field $X$,
\begin{align*}
\langle\nabla_X\nabla_\Ho u_2,X\rangle
={}&\langle[\nabla_v\nabla_X-\nabla_{\nabla_vX}]\nabla f,X\rangle
+|\nabla_Xv|^2+2\langle\nabla_X\nabla f,\Tor^\nabla(v,X)\rangle\\
&+\langle\nabla f,(\nabla_X\Tor^\nabla)(v,X)\rangle
-\langle\Tor^\nabla(X,\Tor^\nabla(X,v)),\nabla f\rangle\\
&+\langle\mathrm{Riem}^\nabla(X,v)\nabla_\Ho f,X\rangle.
\end{align*}
To treat the last torsion term, write
$\Tor^\nabla(X_i,v)=A_i+B_i$ with $A_i\in\Ho$ and $B_i\in\V$.
The matrix $\langle A_i,X_j\rangle$ is symmetric, whereas
$\langle\Tor^\nabla(X_i,X_j),v\rangle$ is skew-symmetric. Therefore
\[
\sum_i\langle\Tor^\nabla(X_i,A_i),v\rangle=0.
\]
The other mixed torsion symmetry gives
$\langle\Tor^\nabla(X_i,B_i),v\rangle=|B_i|^2$, and hence
\[
\sum_i\langle\Tor^\nabla(X_i,\Tor^\nabla(X_i,v)),v\rangle
=\sum_i|\Tor^\nabla(X_i,v)_\V|^2.
\]
Summing over $i$ and applying Lemmas~\ref{partial hessian} and
\ref{mean curvature part} concludes the proof.
	\end{proof}

    \subsection{Bochner's formula for the horizontal Laplacian of the full gradient}

Adding the horizontal and vertical Bochner formulas gives an identity for the full gradient. For $U,V\in\mathcal X(TM)$, define
\begin{align}\label{definition effective curvature}
\mathfrak R(U,V):={}&-\langle U,\delta_\Ho\Tor^\nabla(V)\rangle
-2(\Tor^\nabla(U),\Tor^\nabla(V_\V))_\Ho \notag\\
&-(\Tor^\nabla(U),\Tor^\nabla(V_\Ho))_\Ho
-\tau(U,V_\Ho)+\mathrm{Ric}_\Ho^\nabla(U,V_\Ho) \notag\\
&+\nabla^{\mathrm{sym}}\mathrm H(U,V)
+\langle\Tor^\nabla(\mathrm H,U),V\rangle \notag\\
&+\frac14(J_{U_\V},J_{V_\V})_\Ho
+\sum_{i=1}^n\langle\Tor^\nabla(U_\V,X_i)_\Ho,
\Tor^\nabla(V_\V,X_i)_\Ho\rangle.
\end{align}
    
In general, $\mathfrak R$ is not symmetric. Only its quadratic form $\mathfrak R(U,U)$ will be used below. The last two terms in \eqref{definition effective curvature} are nonnegative on the diagonal and are retained in all curvature lower bounds below.

	\begin{corollary}
		Let $f \in C^\infty (M)$ and  $X_1,\cdots,X_n$ be a local orthonormal frame of horizontal vector fields. We have
		\begin{align*}
			\frac{1}{2}\Delta_{\mathcal{H}} | \nabla f|^2
={}&\langle\nabla f,\nabla\Delta_{\mathcal H}f\rangle
+\mathfrak R(\nabla f,\nabla f)\\
&+\sum_{i,j=1}^n\left(\mathrm{Hess}^{\nabla,\mathrm{sym}}f(X_i,X_j)
+\langle\Tor^\nabla(\nabla_\V f,X_i),X_j\rangle\right)^2\\
&+\sum_{i=1}^n|\nabla_{X_i}\nabla_\V f
-\Tor^\nabla(X_i,\nabla f)_\V|^2.
		\end{align*}
	\end{corollary}
	\begin{proof}
Write $q=\nabla f$, $u=\nabla_\Ho f$, $v=\nabla_\V f$, and
$B_i=\Tor^\nabla(q,X_i)$. Adding the two Bochner formulas gives
\begin{align*}
\frac12\Delta_\Ho|\nabla f|^2-\langle\nabla f,\nabla\Delta_\Ho f\rangle
={}&|\mathrm{Hess}_\Ho^{\nabla,\mathrm{sym}}f|^2
+\sum_i|\nabla_{X_i}v|^2+2\sum_i\langle\nabla_{X_i}q,B_i\rangle\\
&-\langle q,\delta_\Ho\Tor^\nabla(q)\rangle
+\mathrm{Ric}_\Ho^\nabla(q,u)
-(\Tor^\nabla(q),\Tor^\nabla(v))_\Ho\\
&+\sum_i|\Tor^\nabla(v,X_i)_\Ho|^2
+\frac14(J_v,J_v)_\Ho-\tau(q,u)\\
&+\nabla^{\mathrm{sym}}\mathrm H(q,q)
+\langle\Tor^\nabla(\mathrm H,q),q\rangle.
\end{align*}
Since the matrix $\langle\Tor^\nabla(v,X_i),X_j\rangle$ is symmetric,
completion of squares yields
\begin{align*}
&|\mathrm{Hess}_\Ho^{\nabla,\mathrm{sym}}f|^2
+2\sum_i\langle\nabla_{X_i}u,B_i\rangle\\
&\quad=\sum_{i,j}\left(\mathrm{Hess}^{\nabla,\mathrm{sym}}f(X_i,X_j)
+\langle\Tor^\nabla(v,X_i),X_j\rangle\right)^2
-\sum_i|(B_i)_\Ho|^2,
\end{align*}
and
\[
\sum_i|\nabla_{X_i}v|^2+2\sum_i\langle\nabla_{X_i}v,B_i\rangle
=\sum_i|\nabla_{X_i}v+(B_i)_\V|^2-\sum_i|(B_i)_\V|^2.
\]
The two square completions subtract $\sum_i|B_i|^2$. Combining this with
$-(\Tor^\nabla(q),\Tor^\nabla(v))_\Ho$ gives
\[
-2(\Tor^\nabla(q),\Tor^\nabla(v))_\Ho
-(\Tor^\nabla(q),\Tor^\nabla(u))_\Ho.
\]
The remaining quadratic terms, including the two nonnegative torsion terms
involving $v$, are precisely $\mathfrak R(q,q)$ as defined in
\eqref{definition effective curvature}. This proves the identity.
\end{proof}

    As a consequence we get our first curvature dimension estimate:

    \begin{corollary}\label{Bochner's inequality}
Let $f \in C^\infty (M)$. Then,
\[
\frac{1}{2}\Delta_{\mathcal{H}} | \nabla f|^2-\langle \nabla f , \nabla \Delta_{\mathcal{H}} f \rangle \ge \frac{1}{n} (\Delta_\Ho f +\mathrm Hf+\iota (\nabla_\V f))^2+\mathfrak{R} (\nabla f,\nabla f).
\]
    \end{corollary}

    \begin{proof}
Since $\operatorname{tr}_\Ho\mathrm{Hess}^{\nabla,\mathrm{sym}}f=\Delta_\Ho f+\mathrm Hf$, this follows by discarding the mixed-derivative square in the preceding identity and using the lower bound
\begin{align*}
&\sum_{i,j=1}^n \left( \mathrm{Hess}^{\nabla, \mathrm{sym}} f (X_i , X_j)+\langle \Tor^\nabla(\nabla_\V f,X_i),X_j \rangle \right)^2\\
&\qquad\ge\frac{1}{n} \left( \sum_{i=1}^n \left[\mathrm{Hess}^{\nabla, \mathrm{sym}}f (X_i , X_i)+\langle \Tor^\nabla(\nabla_\V f,X_i),X_i \rangle\right] \right)^2.
\end{align*}
    \end{proof}

    \subsection{Curvature dimension inequalities}

    In relation to Bakry-\'Emery calculus let us introduce the following notations: For $f,g \in C^\infty (M)$, we define
    \[
    \Gamma^\Ho_2 (f,g)=\frac{1}{2}\left( \Delta_\Ho \left\langle \nabla_\Ho f, \nabla_\Ho g \right\rangle -\left\langle \nabla_\Ho \Delta_\Ho f, \nabla_\Ho g \right\rangle - \left\langle \nabla_\Ho  f, \nabla_\Ho \Delta_\Ho g \right\rangle \right)
    \]
  and  
	\[
    \Gamma^\V_2 (f,g)=\frac{1}{2}\left( \Delta_\Ho \left\langle \nabla_\V f, \nabla_\V g \right\rangle -\left\langle \nabla_\V \Delta_\Ho f, \nabla_\V g \right\rangle - \left\langle \nabla_\V  f, \nabla_\V \Delta_\Ho g \right\rangle \right).
    \]
The first estimate we get is quite general and follows easily from Corollary \ref{Bochner's inequality} and the  inequality
\[
(a+b)^2 \ge \frac{\lambda}{1+\lambda}a^2 -\lambda b^2.
\]

    \begin{proposition}\label{CD with R}
        Let $\lambda \ge 0$. Assume that there exists a constant $K \in \mathbb{R}$ such that for every $X\in \mathcal{X}(TM)$
        \[
        \mathfrak{R} (X,X) -\frac{\lambda}{n}\beta(X)^2 \ge K | X|^2.
        \]
        Then, for every $f\in C^\infty (M)$
        \[
        \Gamma^\Ho_2 (f,f)+ \Gamma_2^\V (f,f) \ge \frac{\lambda}{n(1+\lambda)} (\Delta_\Ho f)^2 +K |\nabla f|^2.
        \]
    \end{proposition}
    
    \begin{remark}
        If $\mathrm H=0$ and $\iota=0$ (as for the contact metric in Example~\ref{Example contact}; see \cite{BW14}) and $\mathfrak{R} (X,X) \ge K | X|^2$, then one can take $\lambda=+\infty$ so that
        \[
        \Gamma^\Ho_2 (f,f)+ \Gamma_2^\V (f,f) \ge \frac{1}{n} (\Delta_\Ho f)^2 +K |\nabla f|^2.
        \]
    \end{remark}

    The second curvature dimension estimate requires more conditions but can lead to further results. These conditions are for instance satisfied if the manifold $M$ is compact or a Lie group for which both the metric and the foliation are left invariant or both are right invariant.

    \begin{theorem}\label{General CD}
Assume that there is a constant $C \ge 0 $ such that
\[
\max \left\{ | \Tor^\nabla |, | \delta_\Ho \Tor^\nabla |, | \mathrm H| , |\nabla^{\mathrm{sym}} \mathrm{H} |, | \mathrm{Ric}_\Ho^\nabla| \right\} \le C.
\]
Then, there exist constants $\rho_1 \in \mathbb{R}$, $\rho_2 \ge 0,\rho_3 \ge 0,\rho_4 \ge 0,\kappa \ge 0$ and $N \ge n$, depending only on $C,n,m$, such that for every $f \in C^\infty (M)$ and $\nu>0$
\begin{align}\label{generalized CD}
\Gamma^\Ho_2 (f,f)+ \nu \Gamma_2^\V (f,f) \ge \frac{1}{N}(\Delta_\Ho f)^2 +\left(\rho_1-\frac{\kappa}{\nu} \right) | \nabla_\Ho f|^2+ (\rho_2 -\rho_3 \nu -\rho_4 \nu^2) | \nabla_\V f|^2.
\end{align}
Moreover, if the horizontal distribution $\Ho$ is uniformly step-two generating in the sense that there exists a constant $K >0$ such that for every local horizontal orthonormal frame $X_i$ and every $U \in \mathcal{X}(\V)$,
\[
\sum_{i,j=1}^n \left\langle [X_i,X_j], U \right\rangle^2 \ge K | U|^2,
\]
then one can take $\rho_2=K/4>0$.
    \end{theorem}

    \begin{proof}
Let $u=\nabla_\Ho f$, $v=\nabla_\V f$, and $\nu>0$. Combining the horizontal and vertical Bochner formulas and completing the squares as above gives
\begin{align*}
\Gamma_2^\Ho(f,f)+\nu\Gamma_2^\V(f,f)
\ge{}&\frac{1}{n}\bigl(\Delta_\Ho f+\mathrm Hf+\nu\iota(v)\bigr)^2
+\frac14(J_v,J_v)_\Ho
-\nu^2\sum_{i=1}^n|\Tor^\nabla(v,X_i)_\Ho|^2\\
&-\frac1\nu\sum_{i=1}^n|\Tor^\nabla(u+\nu v,X_i)_\V|^2
-C_0\bigl(|u|^2+(1+\nu)|u||v|+\nu|v|^2\bigr),
\end{align*}
where $C_0$ depends only on $C,n,m$. Indeed, the remaining terms in the Bochner formulas are quadratic in $u,v$; the terms independent of $\nu$ contain either two horizontal factors or one horizontal and one vertical factor, and the terms proportional to $\nu$ contain at least one vertical factor. The stated tensor bounds therefore give the last line. Here we have used
\[
\operatorname{tr}_\Ho\mathrm{Hess}^{\nabla,\mathrm{sym}}f
=\Delta_\Ho f+\mathrm Hf.
\]
The elementary inequalities
\[
2|u||v|\le\nu^{-1}|u|^2+\nu|v|^2,
\qquad
2\nu|u||v|\le |u|^2+\nu^2|v|^2,
\]
and
\[
(a+b+c)^2\ge\tfrac12a^2-2b^2-2c^2
\]
now imply that, for constants $C_1,C_2\ge0$ depending only on $C,n,m$,
\begin{align*}
\Gamma_2^\Ho(f,f)+\nu\Gamma_2^\V(f,f)
\ge{}&\frac{1}{2n}(\Delta_\Ho f)^2+\frac14(J_v,J_v)_\Ho\\
&-C_1(1+\nu^{-1})|u|^2-C_2(\nu+\nu^2)|v|^2.
\end{align*}
This proves \eqref{generalized CD}, with $N=2n$, $\rho_1=-C_1$, $\kappa=C_1$, $\rho_2=0$, and $\rho_3=\rho_4=C_2$.

Finally, under the uniform step-two generating condition,
\[
(J_v,J_v)_\Ho
=\sum_{i,j=1}^n\langle\Tor^\nabla(X_i,X_j),v\rangle^2
=\sum_{i,j=1}^n\langle[X_i,X_j],v\rangle^2
\ge K|v|^2.
\]
Thus one can take $\rho_2=K/4>0$.
\end{proof}

 It is clear from the proof that the parameters $N,\kappa,\rho_i$ are not unique. However, if the curvature dimension inequality \eqref{generalized CD} holds then several tensorial constraints are satisfied.

    \begin{corollary}\label{General CD converse}
Assume that \eqref{generalized CD} holds. Then, for every $U\in\mathcal X(\Ho)$,
\begin{align*}
\mathrm{Ric}_\Ho^\nabla(U,U)+\nabla^{\mathrm{sym}}\mathrm H(U,U)
-\tau(U,U)-\frac1N\langle\mathrm H,U\rangle^2
&\ge\rho_1|U|^2,\\
(\Tor^\nabla(U),\Tor^\nabla(U))_\Ho&\le\kappa|U|^2.
\end{align*}
For $V\in\mathcal X(\V)$, set
\begin{align*}
a(V)&=\sum_{i=1}^n|\Tor^\nabla(V,X_i)_\Ho|^2,\\
c(V)&=2\sum_{i=1}^n|\Tor^\nabla(V,X_i)_\V|^2
+\langle V,\delta_\Ho\Tor^\nabla(V)\rangle
-\langle\Tor^\nabla(\mathrm H,V),V\rangle.
\end{align*}
Then, for every $\nu>0$,
\begin{align*}
\frac14(J_V,J_V)_\Ho-\rho_2|V|^2
+\nu\bigl(\rho_3|V|^2-c(V)\bigr)
+\nu^2\left(\rho_4|V|^2-a(V)-\frac{\iota(V)^2}{N}\right)
\ge0.
\end{align*}
In particular,
\[
\frac14(J_V,J_V)_\Ho\ge\rho_2|V|^2,
\qquad
a(V)+\frac{\iota(V)^2}{N}\le\rho_4|V|^2.
\]
\end{corollary}
\begin{proof}
Fix $x\in M$, $u\in\Ho_x$, $v\in\V_x$, and $\nu>0$. By prescribing the first and second derivatives at $x$, choose a smooth function $f$ such that
\begin{align*}
\nabla_\Ho f=u,\qquad \nabla_\V f=v,\qquad
\nabla_{X_i}\nabla_\V f&=-\nu^{-1}\Tor^\nabla(u+\nu v,X_i)_\V,\\
\mathrm{Hess}^{\nabla,\mathrm{sym}}f(X_i,X_j)
&=-\nu\langle\Tor^\nabla(v,X_i),X_j\rangle.
\end{align*}
The last prescription is symmetric in $i,j$ by the mixed torsion symmetry. At $x$ one has
\[
\Delta_\Ho f=-\langle\mathrm H,u\rangle-\nu\iota(v).
\]
The completed squares in the proof of Theorem~\ref{General CD} vanish. Taking $v=0$ in the Bochner identities and applying \eqref{generalized CD} gives
\begin{align*}
\mathrm{Ric}_\Ho^\nabla(u,u)+\nabla^{\mathrm{sym}}\mathrm H(u,u)-\tau(u,u)
-\frac1\nu(\Tor^\nabla(u),\Tor^\nabla(u))_\Ho
\ge\frac{\langle\mathrm H,u\rangle^2}{N}
+\left(\rho_1-\frac\kappa\nu\right)|u|^2.
\end{align*}
Letting $\nu\to\infty$, and then multiplying by $\nu$ and letting $\nu\to0$, proves the horizontal assertions. Taking $u=0$ gives
\[
\frac14(J_v,J_v)_\Ho-\nu c(v)-\nu^2a(v)
\ge\frac{\nu^2\iota(v)^2}{N}
+(\rho_2-\rho_3\nu-\rho_4\nu^2)|v|^2,
\]
which proves the remaining claims. Notice that this polynomial inequality does not in general imply a separate bound on its linear coefficient.
\end{proof}

	\section{Applications}

    We now turn to the second part of the paper and focus on geometric analysis applications of the Bochner's identities.

    \subsection{Horizontal Laplacian comparison theorem}

    The first application of the curvature dimension estimates is the generalization to our setting of the horizontal Laplacian comparison theorem proved in \cite{2025arXiv250913276B}. More precisely, the result below removes the bundle-like condition on the metric and the minimality of the leaves from the assumptions of  \cite{2025arXiv250913276B}.

\begin{theorem}\label{Comparison horizontal Laplacian}
Let $\lambda >0 $. Assume that there exists a constant $K \in \mathbb{R}$ such that for every $X\in \mathcal{X}(TM)$
        \[
        \mathfrak{R} (X,X) -\frac{\lambda}{n}\beta(X)^2 \ge K | X|^2.
        \]
        Let $p \in M$ and denote $r_p(x)=d(p,x)$. Then, for $x\neq p$ not in the  cut-locus of $p$,
\begin{align*}
\Delta_\Ho r_p (x) \le 
\begin{cases}
 \sqrt { NK} \cot\left(\sqrt { \frac{K}{N}} r_p(x) \right) &\text{if}\ K>0,
\\
\displaystyle\frac{N}{r_p(x)} &\text{if}\ K = 0,
\\
 \sqrt{N |K|} \coth\left(\sqrt{\frac{|K|}{N}} r_p(x)\right) &\text{if}\ K<0,
\end{cases}
\end{align*}
where $N=\frac{n(1+\lambda)}{\lambda}$.
\end{theorem}

\begin{proof}
Let $p \in M$ and $x \neq p$ not in the cut-locus of $p$. Let $\gamma:[0,r_p(x)]\to M$ be the unique length-parametrized geodesic from $p$ to $x$. Denote $\phi (t)=\Delta_\Ho r_p (\gamma (t))$. We have then
\begin{align*}
    \phi'(t)&=\left\langle \gamma'(t), \nabla \Delta_\Ho r_p (\gamma (t)) \right\rangle \\
     &=\left\langle \nabla r_p (\gamma(t)), \nabla \Delta_\Ho r_p (\gamma (t)) \right\rangle.
\end{align*}
From Proposition \ref{CD with R} we have
\begin{align*}
 & \frac{1}{2}\Delta_{\mathcal{H}} | \nabla r_p|^2(\gamma(t))-\langle \nabla r_p (\gamma(t)) , \nabla \Delta_{\mathcal{H}} r_p (\gamma(t)) \rangle \\
\ge & \frac{\lambda}{n(1+\lambda)} (\Delta_\Ho r_p (\gamma(t)))^2 +K |\nabla r_p (\gamma(t))|^2.
\end{align*}
Since $| \nabla r_p|=1$ we deduce
\[
-\phi'(t) \ge \frac{\lambda}{n(1+\lambda)} \phi(t)^2+K.
\]
Let us denote
\begin{align*}
G(t) = 
\begin{cases}
  \sin \left(\sqrt { \frac{K}{N}} t \right) &\text{if}\ K>0,
\\
\displaystyle t &\text{if}\ K = 0,
\\
 \sinh\left(\sqrt{\frac{|K|}{N}} t\right) &\text{if}\ K<0.
\end{cases}
\end{align*}
When $K>0$, we first work for $r_p(x)<\pi\sqrt{N/K}$, so that $G$ is positive along the geodesic. We use the elementary inequality
\[
 \frac{\lambda}{n(1+\lambda)} \phi(t)^2 \ge 2 \frac{G'(t)}{G(t)} \phi(t)-\frac{n(1+\lambda)}{\lambda} \frac{G'(t)^2}{G(t)^2},
\]
which yields
\[
-\phi'(t) \ge 2 \frac{G'(t)}{G(t)} \phi(t)-\frac{n(1+\lambda)}{\lambda} \frac{G'(t)^2}{G(t)^2} +K.
\]
Multiplying by $G(t)^2$ and integrating from $0$ to $r_p(x)$ one obtains
\[
-\int_0^{r_p(x)} \left(\phi'(t)G(t)^2 +2 G'(t) G(t) \phi(t)\right) dt \ge  \int_0^{r_p(x)} \left(-\frac{n(1+\lambda)}{\lambda}G'(t)^2+K G(t)^2\right) dt.
\]
Now, it is clear from a local computation in Riemannian exponential coordinates that $\lim_{t \to 0} \phi(t) G(t)^2 =0$, therefore one has
\[
-G(r_p(x))^2 \phi(r_p(x)) \ge  \int_0^{r_p(x)} \left(-\frac{n(1+\lambda)}{\lambda}G'(t)^2+K G(t)^2\right) dt
\]
Evaluating the integral proves the comparison. If $K>0$, a minimizing geodesic cannot remain outside the cut locus up to and beyond $\pi\sqrt{N/K}$: the upper bound tends to $-\infty$ at that time, contradicting smoothness of $\Delta_\Ho r_p$ there. Thus every point outside the cut locus has $r_p(x)<\pi\sqrt{N/K}$, as required.
\end{proof}

Following \cite{2025arXiv250913276B}, we obtain the following consequences of the Laplacian comparison theorem.

\begin{corollary}[Bonnet-Myers type theorem]\label{BMyers}
Let $\lambda > 0$. Assume that there exists a constant $K >0$ such that for every $X \in  \mathcal{X}(TM)$
        \[
        \mathfrak{R} (X,X) -\frac{\lambda}{n}\beta(X)^2 \ge K | X|^2
        \]
then $M$ is compact and 
\[
\mathbf{diam} (M ) \le \pi \sqrt{ \frac{n(1+\lambda)}{\lambda K}}.
\]
\end{corollary}

\begin{corollary}
 Let $\lambda  >0$. Assume that there exists a constant $K \le 0$ such that for every $X \in  \mathcal{X}(TM)$
        \[
        \mathfrak{R} (X,X) -\frac{\lambda}{n}\beta(X)^2 \ge K | X|^2.
        \]
Then  the heat  semigroup $P_t$ is stochastically complete meaning that for every $x \in M$ and $t \ge 0$
\[
P_t 1(x)=1.
\] 
Moreover, there exist  constants $c_1,c_2,c_3>0$ such that for every $x \in M$ and $t>0$
\begin{align*}
p_t(x,x) \ge
\begin{cases}
 \frac{c_1}{\mu(B(x ,c_2\sqrt{t}))}, \text{ if } K=0, \\
 \frac{c_1}{\mu\left(B\left(x ,c_2\left(\frac{e^{c_3\sqrt{|K|}t}-1}{ \sqrt{|K|}} \right)^{1/2}\right)\right)}, \text{ if } K<0.
 \end{cases}
\end{align*}
Here $B(x,r)$ denotes the ball with center $x$ and radius $r$ for the Riemannian distance.
\end{corollary}
\begin{proof}
Set $k=-K\ge0$ and $r=d(x,\cdot)$. The comparison theorem gives, in the
distributional sense,
\[
\Delta_\Ho r^2\le2N+2+2\sqrt{Nk}\,r\le a+b r^2,
\qquad a=2N+2+\sqrt{Nk},\quad b=\sqrt{Nk}.
\]
Here the cut locus contributes a nonpositive singular part to the Hessian.
Applying the stopped It\^o inequality, justified by local semiconcave
approximation as in the localization lemma below, and Gronwall's lemma
yields
\[
\mathbb E_x r(X_{t\wedge\tau_R})^2\le\Psi(t),\qquad
\Psi(t)=
\begin{cases}
at,&b=0,\\
a(e^{bt}-1)/b,&b>0,
\end{cases}
\]
where $\tau_R$ is the first exit time from $B(x,R)$.
Thus $\mathbb P_x(\tau_R\le t)\le\Psi(t)/R^2$, which proves stochastic
completeness. Letting $R\to\infty$ and using Markov's inequality gives
\[
P_{t/2}\mathbf1_{B(x,\sqrt{2\Psi(t/2)})}(x)\ge\tfrac12.
\]
Symmetry of the heat kernel and the Cauchy--Schwarz inequality now imply
\[
p_t(x,x)\ge
\frac{1}{4\mu(B(x,\sqrt{2\Psi(t/2)}))}.
\]
This has the asserted form, with constants independent of $x$ and $t$.
\end{proof}
A positive curvature lower bound also implies the hypothesis of this corollary with $K=0$.

\subsection{First eigenvalue estimates}

The curvature dimension inequalities also imply estimates for the first eigenvalue of the horizontal Laplacian.

\begin{proposition}
Assume that $M$ is compact and that there exists a constant $K >0$ such that for every $X \in \mathcal{X}(TM)$
        \[
        \mathfrak{R} (X,X)  \ge K | X|^2.
        \]
    Then the first nonzero eigenvalue $\lambda_1$ of $-\Delta_\Ho$ satisfies 
    \[
    \lambda_1 \ge K.
    \]
\end{proposition}
\begin{proof}
By the full gradient Bochner formula, for every smooth $f$,
\[
\frac12\Delta_\Ho|\nabla f|^2
\ge\langle\nabla f,\nabla\Delta_\Ho f\rangle+\mathfrak R(\nabla f,\nabla f).
\]
Let $f$ be a nonconstant eigenfunction with $\Delta_\Ho f=-\lambda_1f$.
Integrating on $M$ gives
\[
0\ge-\lambda_1\int_M|\nabla f|^2\,d\mu
+\int_M\mathfrak R(\nabla f,\nabla f)\,d\mu
\ge(K-\lambda_1)\int_M|\nabla f|^2\,d\mu.
\]
The conclusion follows.
\end{proof}

    Possibly better estimates for the first eigenvalue might be obtained  from the one-parameter family of curvature dimension inequalities.

    \begin{proposition}
        Assume that \eqref{generalized CD} is satisfied with $N>1$, $\rho_2>0$, $\rho_3,\rho_4,\kappa\ge0$, and that
        \[
        \rho_1 \rho_2 > \kappa (\rho_3 +\sqrt{\rho_2\rho_4}).
        \]
        Then $M$ is compact and  the first nonzero eigenvalue $\lambda_1$ of $-\Delta_\Ho$ satisfies 
    \[
    \lambda_1 \ge \frac{\rho_1 \rho_2 - \kappa (\rho_3 +\sqrt{\rho_2\rho_4})}{\left(\frac{N-1}{N}\right)\rho_2+\kappa }.
    \]
    \end{proposition}

    \begin{proof}
Assume that \eqref{generalized CD} holds with $\rho_1 \rho_2 > \kappa (\rho_3 +\sqrt{\rho_2\rho_4})$. Then we have
\[
\rho_2 -\rho_3 \frac{\kappa}{\rho_1}-\rho_4 \frac{\kappa^2}{\rho_1^2}>0.
\]
Therefore there exists $\nu >0$ such that
\[
\rho_1-\frac{\kappa}{\nu} >0, \qquad \rho_2-\rho_3 \nu -\rho_4 \nu^2 >0.
\]
This implies there exist $\nu>0$, $N \ge n$  and $K>0$ such that for every $f \in C^\infty (M)$ 
\begin{align*}
\Gamma^\Ho_2 (f,f)+ \nu \Gamma_2^\V (f,f) \ge \frac{1}{N}(\Delta_\Ho f)^2 +K (|\nabla_\Ho f|^2+\nu|\nabla_\V f|^2).
\end{align*}
Using the proof of Theorem \ref{Comparison horizontal Laplacian} for the distance associated to the Riemannian metric
\[
g^\nu(X,Y)=g(X_\Ho,Y_\Ho) +\frac{1}{\nu} g(X_\V,Y_\V)
\]
which is complete since it is uniformly equivalent to $g$, we deduce a Laplacian comparison theorem for this metric. Since $K>0$, the same Bonnet--Myers argument shows that $M$ is compact.

Now, for a non constant $f \in C^\infty(M)$ such that $\Delta_\Ho f =-\lambda f$  integrating the inequality
\[
\Gamma^\Ho_2 (f,f)+ \nu \Gamma_2^\V (f,f) \ge \frac{1}{N}(\Delta_\Ho f)^2 +\left(\rho_1-\frac{\kappa}{\nu} \right) | \nabla_\Ho f|^2+ (\rho_2 -\rho_3 \nu -\rho_4 \nu^2) | \nabla_\V f|^2
\]
yields
\begin{align*}
 & \lambda^2 \int_M f^2 d\mu +\nu \lambda \int_M |\nabla_\V f|^2 d\mu \\
 \ge & \frac{1}{N} \lambda^2   \int_M f^2 d\mu+\lambda \left(\rho_1-\frac{\kappa}{\nu} \right) \int_M f^2 d\mu +(\rho_2 -\rho_3 \nu -\rho_4 \nu^2) \int_M | \nabla_\V f|^2 d\mu.
\end{align*}
We now choose $\nu =\frac{ \rho_2 }{\rho_3+\lambda+\sqrt{\rho_2\rho_4}}$. Then we have
\[
\rho_2 -(\rho_3+\lambda) \nu -\rho_4 \nu^2 \ge 0.
\]
With this choice of $\nu$ we obtain
\begin{align*}
  \lambda^2 \int_M f^2 d\mu  
 \ge  \frac{1}{N} \lambda^2   \int_M f^2 d\mu+\lambda \left(\rho_1-\frac{\kappa}{\nu} \right) \int_M f^2 d\mu .
\end{align*}
This gives
\[
\lambda \ge \frac{N}{N-1} \left(\rho_1-\frac{\kappa}{\nu} \right)
\]
and the conclusion follows easily.
    \end{proof}

\begin{remark}
    When $\rho_3=\rho_4=0$, the estimate becomes
    \[
    \lambda_1\ge\frac{\rho_1\rho_2}{\left(\frac{N-1}{N}\right)\rho_2+\kappa}.
    \]
    We also refer to \cite{BaudoinEMS2014,BergeGrong2019} for eigenvalue estimates under additional geometric hypotheses.
\end{remark}

\subsection{Heat kernel gradient bounds}

In the spirit of the celebrated Li-Yau work \cite{LiYau1986} Bochner's formulas can be used to get gradient bounds on the heat kernel. Thanks to Bakry-\'Emery calculus, such gradient bounds have a wide range of aplications to functionals inequalities, see \cite{BGL}.

For the arguments below, set
\[
\mathcal T(f,g)=\langle\nabla f,\nabla g\rangle,\qquad
\mathcal T_2(f,g)=\tfrac12\bigl(\Delta_\Ho \mathcal T(f,g)-\mathcal T(\Delta_\Ho f,g)-\mathcal T(f,\Delta_\Ho g)\bigr).
\]
We abbreviate $\mathcal T(f)=\mathcal T(f,f)$ and similarly for $\mathcal T_2$.

\begin{localizationlemma}
Let $\mathcal T(f)=|\nabla f|^2$ and assume that, for some
$N\ge n$ and $K\in\mathbb R$,
\[
\mathcal T_2(f)\ge \frac1N(\Delta_\Ho f)^2+K\mathcal T(f),\qquad f\in C^\infty(M).
\]
Then there exists a smooth proper function $W\ge1$ such that
\[
\Delta_\Ho W\le C,\qquad |\nabla W|\le C
\]
for some constant $C$. Moreover, $P_t1=1$ and, for every $f\in C_0^\infty(M)$
and $t>0$,
\[
\sup_{0\le s\le t}\|\nabla P_sf\|_\infty<\infty.
\]
\end{localizationlemma}
\begin{proof}
The proof of Theorem~\ref{Comparison horizontal Laplacian} uses only the
displayed curvature-dimension inequality. With $r=d(p,\cdot)$ and
$k=\max\{-K,0\}$, it gives $\Delta_\Ho r\le N/r+\sqrt{Nk}$ outside the cut locus.
Thus $F=\sqrt{1+r^2}$ is proper and locally semiconcave, $|\nabla F|\le1$,
and $\Delta_\Ho F\le A:=N+\sqrt{Nk}+1$ in the sense of distributions. Indeed, the
singular part of its distributional Hessian is negative semidefinite.

We give the smoothing argument because an upper bound on $\Delta_\Ho F$ must be
preserved. Set $u=-F$. On a relatively compact coordinate neighborhood,
write $\Delta_\Ho =a^{ij}\partial_i\partial_j+b^i\partial_i$, where
$a=\sigma\sigma^{\mathsf T}$ is smooth. For some $B$, the matrix-valued
measure $\mathsf M=D^2u+B I\,dx$ is nonnegative. For $|x-y|\le\varepsilon$ and
$0<\delta<1$,
\[
a(x)\ge(1-\delta)a(y)-C\varepsilon^2\delta^{-1}I.
\]
Since $u$ is locally Lipschitz, convolution with a standard mollifier
$\eta_\varepsilon$ satisfies
$\int\eta_\varepsilon(x-y)\,\operatorname{tr}\mathsf M(dy)\le C/\varepsilon$.
Consequently the local mollifications $u_\varepsilon$ satisfy
\[
\Delta_\Ho u_\varepsilon\ge-A-C(\delta+\varepsilon/\delta+\varepsilon),\qquad
|\nabla u_\varepsilon|\le1+C\varepsilon.
\]
Taking $\delta=\sqrt\varepsilon$ gives local smooth strict approximations
for the convex inequalities $\Delta_\Ho v\ge-A-1$ and $|\nabla v|\le2$.
The Richberg approximation theorem
\cite[Theorem~1.2 and Lemma~3.2]{HLP2020} gives a smooth
$v$ satisfying these inequalities and $u<v<u+1/2$ on $M$.
Then $W=2-v$ has the required properties.

The bound on $\Delta_\Ho W$ and properness of $W$ imply nonexplosion by the stopped
It\^o formula. The cutoff argument of
\cite[Lemma~2.1]{Wang2012}, with
the localized generator $h(W/R)^2\Delta_\Ho$ and $\widetilde\Gamma=\mathcal T$, then proves the asserted
gradient boundedness. This part of that proof uses only
$\mathcal T_2\ge K\mathcal T$ and the exhaustion bounds established above.
\end{proof}

\subsubsection{Bakry-\'Emery type estimates}

\begin{theorem}\label{BE estimate}
    Let $\lambda  >0$. Assume that there exists a constant $K \in \mathbb{R}$ such that for every $X \in  \mathcal{X}(TM)$
        \[
        \mathfrak{R} (X,X) -\frac{\lambda}{n}\beta(X)^2 \ge K | X|^2.
        \]
        Then, for every $f \in C_0^\infty(M)$,
        \[
        | \nabla P_{t}f|^2+\frac{2}{N}\frac{1-e^{-2Kt}}{2K} (\Delta_{\mathcal{H}} P_{t}f)^2 \le e^{-2Kt}P_t (|\nabla f|^2)
        \]
where $N=\frac{n(1+\lambda)}{\lambda}$.        When $K=0$, we understand $\frac{1-e^{-2Kt}}{2K}$ as $t$.
\end{theorem}

\begin{proof}
By Proposition~\ref{CD with R},
\[
\mathcal T_2(f)\ge\frac1N(\Delta_\Ho f)^2+K\mathcal T(f).
\]
The localization lemma therefore applies. Fix $f\in C_0^\infty(M)$ and $t>0$,
and let $(X_s)$ be the diffusion generated by $\Delta_\Ho$. The stopped It\^o formula shows that
\[
e^{-2Ks}\mathcal T(P_{t-s}f)(X_s)
-\frac2N\int_0^s e^{-2Kr}(\Delta_\Ho P_{t-r}f)^2(X_r)\,dr
\]
is a local submartingale for $0\le s\le t$. It is bounded, since the localization
lemma bounds the gradients and $\Delta_\Ho P_rf=P_r\Delta_\Ho f$, with
$\|\Delta_\Ho P_rf\|_\infty\le\|\Delta_\Ho f\|_\infty$. It is therefore a submartingale.
Taking expectations gives
\[
e^{-2Kt}P_t\mathcal T(f)-\mathcal T(P_tf)
\ge\frac2N\int_0^t e^{-2Ks}P_s\bigl((\Delta_\Ho P_{t-s}f)^2\bigr)\,ds.
\]
Since $P_s1=1$, Jensen's inequality and the semigroup property give
$P_s((\Delta_\Ho P_{t-s}f)^2)\ge(\Delta_\Ho P_tf)^2$. Evaluating the integral proves the result.
\end{proof}

\subsubsection{Global regularization estimates}

\begin{theorem}
Assume that there is a constant $C \ge 0 $ such that
\[
\max \left\{ | \Tor^\nabla |, | \delta_\Ho \Tor^\nabla |, | \mathrm H| , |\nabla^{\mathrm{sym}} \mathrm{H} |, | \mathrm{Ric}_\Ho^\nabla| \right\} \le C
\]

Assume moreover that the horizontal distribution $\Ho$ is uniformly step-two generating as in theorem \ref{General CD}. Then there exist constants $c_1,c_2,c_3>0$ and $t_0>0$ such that for every $f \in C_0^\infty(M)$ and $0<t <t_0$
\[
| \nabla_\Ho P_t f |^2 \le \frac{c_1}{t}  \left( P_t(f^2)-(P_t f)^2 \right)
\]
\[
| \nabla_\V P_t f |^2 \le \frac{c_2}{t^2}  \left( P_t(f^2)-(P_t f)^2 \right)
\]
\[
(\Delta_\Ho P_t f)^2 \le  \frac{c_3}{t^2}  \left( P_t(f^2)-(P_t f)^2 \right).
\]
\end{theorem}

\begin{proof}
By Theorem~\ref{General CD}, \eqref{generalized CD} holds with $\rho_2>0$.
Decreasing $\rho_1$ if necessary, we may assume $\rho_1\le0$.
Choose $\nu_0>0$ and $\tilde\rho_2>0$ such that, for $0<\nu\le\nu_0$,
\begin{align}\label{CD local}
\Gamma_2^\Ho(f,f)+\nu\Gamma_2^\V(f,f)
\ge\frac1N(\Delta_\Ho f)^2
+\left(\rho_1-\frac\kappa\nu\right)|\nabla_\Ho f|^2
+\tilde\rho_2|\nabla_\V f|^2.
\end{align}
Taking $\nu=1$ in \eqref{generalized CD}, the localization lemma applies with
$K=\min\{\rho_1-\kappa,\rho_2-\rho_3-\rho_4\}$.
Fix $0<t<t_0:=\nu_0/\tilde\rho_2$ and $f\in C_0^\infty(M)$, and set
\[
F(s,x)=(t-s)|\nabla_\Ho P_{t-s}f|^2(x)
+\tilde\rho_2(t-s)^2|\nabla_\V P_{t-s}f|^2(x).
\]
For $s<t$, a direct computation gives
\begin{align*}
(\partial_s+\Delta_\Ho)F
={}&-|\nabla_\Ho P_{t-s}f|^2
-2\tilde\rho_2(t-s)|\nabla_\V P_{t-s}f|^2\\
&+2(t-s)\Gamma_2^\Ho(P_{t-s}f,P_{t-s}f)
+2\tilde\rho_2(t-s)^2\Gamma_2^\V(P_{t-s}f,P_{t-s}f).
\end{align*}
Using \eqref{CD local} with $\nu=\tilde\rho_2(t-s)$, we obtain
\[
(\partial_s+\Delta_\Ho)F
\ge\frac{2(t-s)}N(\Delta_\Ho P_{t-s}f)^2
+\left(2\rho_1t-\frac{2\kappa}{\tilde\rho_2}-1\right)|\nabla_\Ho P_{t-s}f|^2.
\]
Apply the stopped It\^o formula, first up to $t-\varepsilon$.
The gradient bound in the localization lemma and the bound on $\Delta_\Ho P_rf$
justify removal of the stopping times and then the limit $\varepsilon\downarrow0$.
Since $F(t,\cdot)=0$, Jensen's inequality gives
\begin{align*}
-F(0,x)\ge{}&\frac{t^2}{N}(\Delta_\Ho P_tf)^2(x)\\
&+\left(2\rho_1t-\frac{2\kappa}{\tilde\rho_2}-1\right)
\int_0^tP_s(|\nabla_\Ho P_{t-s}f|^2)(x)\,ds.
\end{align*}
The same stopped It\^o formula applied to $(P_{t-s}f)^2$ yields
\[
2\int_0^tP_s(|\nabla_\Ho P_{t-s}f|^2)\,ds
=P_t(f^2)-(P_tf)^2.
\]
Consequently,
\begin{align*}
t|\nabla_\Ho P_tf|^2+\tilde\rho_2t^2|\nabla_\V P_tf|^2
+\frac{t^2}{N}(\Delta_\Ho P_tf)^2
\le\frac12\left(1+\frac{2\kappa}{\tilde\rho_2}-2\rho_1t\right)
\bigl(P_t(f^2)-(P_tf)^2\bigr).
\end{align*}
The three asserted estimates follow, after decreasing $t_0$ if necessary.
\end{proof}

It is worth noting that the estimate 
\[
| \nabla_\Ho P_t f |^2 \le \frac{C}{t}  \left( P_t(f^2)-(P_t f)^2 \right)
\]

immediately implies the so-called weak Bakry-\'Emery estimate 
\[
\| \nabla_\Ho P_t f \|_{L^\infty} \le \frac{C}{\sqrt{t}} \| f \|_{L^\infty},\qquad 0<t<t_0,
\]
for $f\in C_0^\infty(M)$. Such gradient estimates are useful in the study of Riesz transforms (see \cite{BGRiesz}) and isoperimetric and Sobolev inequalities (see \cite{Alonso-Ruiz-Baudoin-Chen-2020}), under suitable additional global assumptions. See also \cite{DePontiStefani2025} for further discussions on this regularization property.

On the other hand, the estimate 
\[
(\Delta_\Ho P_t f)^2 \le  \frac{c}{t^2}  \left( P_t(f^2)-(P_t f)^2 \right)
\]
is more related to second order Riesz transforms, see \cite{2021arXiv210813058C}.

\section{Beyond the foliated setting}

Using the same methods and computations, our results can in fact be extended beyond the foliated setting. Consider a complete Riemannian manifold $(M,g)$ whose tangent bundle admits an orthogonal decomposition
\[
TM=\mathcal{H}\oplus\mathcal{V},
\]
where $\mathcal{H}$ and $\mathcal{V}$ are smooth sub-bundles and $\Ho$ is bracket-generating. The adapted connection $\nabla$ is still well defined in this context (see Remark~\ref{involutivity assumption}). The only difference with the foliated case is that one no longer necessarily has
\[
\Tor^\nabla(\mathcal{V},\mathcal{V})=0,
\]
but only the weaker condition $\Tor^\nabla(\mathcal{V},\mathcal{V})\subset\mathcal{H}$. More precisely, the last line of \eqref{formula torsion} is replaced by $\Tor^\nabla(U,V)=-[U,V]_\Ho$ for vertical $U,V$.

The horizontal Laplacian $\Delta_{\mathcal{H}}$ is again defined as the generator of the horizontal Dirichlet form
\[
\mathcal{E}_{\mathcal{H}}(f,f)=\int_M |\nabla_{\mathcal{H}}f|^2\,d\mu,
\]
and admits the local decomposition
\[
\Delta_{\mathcal{H}}
=
\sum_{i=1}^n X_i^2
-\sum_{i=1}^n (D_{X_i}X_i)_{\mathcal{H}}
-\sum_{\ell=1}^m (D_{Z_\ell}Z_\ell)_{\mathcal{H}},
\]
where $(X_i)_{1\le i\le n}$ and $(Z_\ell)_{1\le \ell\le m}$ are local orthonormal frames of $\mathcal{H}$ and $\mathcal{V}$, respectively.

A notable difference is that the vector field
\[
\mathrm{H}=\sum_{\ell=1}^m (D_{Z_\ell}Z_\ell)_{\mathcal{H}}
\]
can no longer be interpreted as the mean curvature vector field of the leaves, since $\mathcal{V}$ is not assumed to be integrable. Nevertheless, $\mathrm{H}$ still admits a natural geometric interpretation in terms of the connection $\nabla$. Indeed, using \eqref{Levi-Civita}, one readily checks that
\[
\mathrm{H}=\sum_{\ell=1}^m J_{Z_\ell}Z_\ell.
\]

Indeed, $\nabla_ZZ=D_ZZ-J_ZZ$ and
$\langle J_ZZ,Y\rangle=\langle Z,\Tor^\nabla(Z,Y)\rangle=0$
for vertical $Y,Z$. Taking horizontal projections gives the identity.
The Bochner computations use preservation of the splitting, the mixed torsion
symmetries, and $\Tor^\nabla(\Ho,\Ho)\subset\V$; the torsion contraction in
the proof of Theorem~\ref{Bochner} does not require integrability of $\V$.
Consequently, the general tensorial statements of Sections~3 and~4 and their
analytic consequences remain valid, with $\mathrm H$ defined above.
Statements involving the geometry of the leaves are restricted to the foliated setting. This slightly more general framework is, for instance, well suited to the study of curvature-dimension inequalities in the context of quaternionic contact manifolds. We have chosen to present our results  in the setting of foliations, as the geometric interpretation of the various assumptions and tensorial quantities, such as the bundle-like condition, total geodesicity, and mean curvature of the leaves, is more transparent in that well-established framework.

\end{document}